\documentclass[a4paper]{amsart}
\usepackage{geometry}

\usepackage{mathtext}
\usepackage[T1]{fontenc}
\usepackage{color}
\usepackage{amsmath}
\usepackage{amsfonts,amssymb,amsthm}
\usepackage[pdftex]{graphicx}
\usepackage{comment}
\usepackage{float}
\usepackage{todonotes}
\usepackage{microtype}

\usepackage{caption}
\usepackage[colorlinks=false, linktocpage=true]{hyperref}

\DeclareRobustCommand{\SkipTocEntry}[5]{}

\usepackage{tikz}
\usetikzlibrary{positioning, arrows.meta}
\usetikzlibrary{decorations.markings}
\tikzset{middlearrow/.style={
        decoration={markings,
            mark= at position 0.5 with {\arrow{#1}} ,
        },
        postaction={decorate}
    }
}
\tikzset{firstthirdarrow/.style={
        decoration={markings,
            mark= at position 0.33 with {\arrow{#1}} ,
        },
        postaction={decorate}
    }
}
\tikzset{secondthirdarrow/.style={
        decoration={markings,
            mark= at position 0.66 with {\arrow{#1}} ,
        },
        postaction={decorate}
    }
}
\usetikzlibrary{cd}

\usepackage{enumerate}

\usepackage{cite}

\usepackage{color}
\definecolor{NoteColor}{rgb}{1,0,0}

\DeclareMathOperator{\sgn}{sgn}

\DeclareMathOperator{\Aut}{Aut}

\DeclareMathOperator{\Hom}{Hom}
\DeclareMathOperator{\diag}{diag}
\DeclareMathOperator{\Id}{Id}
\DeclareMathOperator{\Sym}{Sym}
\DeclareMathOperator{\Span}{Span}

\DeclareMathOperator{\Mat}{Mat}
\DeclareMathOperator{\PSp}{PSp}
\DeclareMathOperator{\Sp}{Sp}
\DeclareMathOperator{\GL}{GL}
\DeclareMathOperator{\SL}{SL}
\DeclareMathOperator{\PSL}{PSL}
\DeclareMathOperator{\PO}{PO}
\DeclareMathOperator{\OO}{O}
\DeclareMathOperator{\UU}{U}

\DeclareMathOperator{\SO}{SO}

\DeclareMathOperator{\Rep}{Rep}
\DeclareMathOperator{\Stab}{Stab}

\DeclareMathOperator{\Lie}{Lie}

\DeclareMathOperator{\Fix}{Fix}

\DeclareMathOperator{\Herm}{Herm}

\DeclareMathOperator{\Flag}{Flag}

\DeclareMathOperator{\Conf}{Conf}
\DeclareMathOperator{\fr}{fr}

\DeclareMathOperator{\Norm}{Norm}
\DeclareMathOperator{\Ad}{Ad}
\DeclareMathOperator{\SU}{SU}
\DeclareMathOperator{\PU}{PU}
\DeclareMathOperator{\PGL}{PGL}

\newcommand{\R}{\mathbb R}

\newcommand{\CC}{\mathbb C}
\newcommand{\HH}{\mathbb H}

\newcommand{\N}{\mathbb N}
\newcommand{\Z}{\mathbb Z}

\newcommand{\Oc}{\mathbb O}
\newcommand{\T}{\mathcal T}
\newcommand{\Pp}{\mathcal P^+}
\newcommand{\Pm}{\mathcal P^-}
\newcommand{\Up}{\mathcal U^+}
\newcommand{\Um}{\mathcal U^-}
\newcommand{\Upp}{\mathcal U^+_{>0}}

\newcommand{\up}{\mathfrak u^+}
\newcommand{\um}{\mathfrak u^-}
\newcommand{\F}{\mathcal F}
\renewcommand{\L}{\mathcal L}
\renewcommand{\T}{\mathcal T}

\newcommand{\g}{\mathfrak g}
\newcommand{\s}{\mathfrak s}
\renewcommand{\sl}{\mathfrak{sl}}
\newcommand{\coloneqq}{\mathrel{\mathop:}=}

\theoremstyle{plain}
\newtheorem{teo}{Theorem}[section]

\newtheorem{cor}[teo]{Corollary}

\newtheorem{lem}[teo]{Lemma}

\newtheorem{prop}[teo]{Proposition}

\newtheorem{fact}[teo]{Fact}

\theoremstyle{definition}
\newtheorem{df}[teo]{Definition}

\theoremstyle{remark}
\newtheorem{rem}[teo]{Remark}

\newcommand{\bs}{\smallsetminus}
\newcommand{\defin}{\emph}

\renewcommand{\emptyset}{\varnothing}

\begin{document}

\sloppy

\title{Parametrizing spaces of positive representations}

\author[O. Guichard]{Olivier Guichard}
\address{IRMA UMR 7501, Université de
  Strasbourg et CNRS,
  Strasbourg, France}
\email{olivier.guichard@math.unistra.fr}

\author[E. Rogozinnikov]{Eugen Rogozinnikov}
\address{IRMA UMR 7501, Université de
  Strasbourg et CNRS,
  Strasbourg, France}
\email{erogozinnikov@gmail.com}

\author[A. Wienhard]{Anna Wienhard}
\address{Mathematisches Institut, Ruprecht-Karls-Universität Heidelberg,
  Germany\hfill{}\ \linebreak
  HITS gGmbH, Heidelberg Institute for Theoretical Studies, Heidelberg,
  Germany}
\email{wienhard@uni-heidelberg.de}

\thanks{E.R.\;thanks the Labex IRMIA of the Universit\'e de Strasbourg for support during the preparation of this article. O.G.\;thanks the Institut univesitaire de France and acknowledges support of the Agence Nationale de la Recherche under the grant DynGeo (ANR-16-CE40-0025).
A.W.\;is partially funded by the European Research Council under  ERC-Advanced Grant 101018839, by the Klaus Tschira Foundation, and by the Deutsche Forschungsgemeinschaft under Germany's Excellence Strategy EXC-2181/1 - 390900948 (the Heidelberg STRUCTURES Cluster of Excellence)}

\begin{abstract}
 Using Lusztig's total positivity in split real Lie groups V.\;Fock and A.\;Goncharov have introduced spaces of positive (framed) representations. 
For general semisimple Lie groups a generalization of Lusztig's total positivity was recently introduced by O.\;Guichard and A.\;Wienhard. They also introduced the associated space of positive representations. Here we consider the corresponding spaces of positive framed representations of the fundamental group of a punctured surface. 
We give several parametrizations of the spaces of framed positive representations. Using these parametrizations, we describe their topology and their homotopy type. We show that the number of connected components of the space of framed positive representations agrees with the number of connected components of the space of positive representations, and determine this number for simple Lie groups. Along the way, we also parametrize, for  an arbitrary semisimple Lie group, the space of representations of the fundamental group of a punctured surface which are transverse with respect to a fixed ideal triangulation of the surface.
\end{abstract}

\maketitle

\tableofcontents

\section{Introduction}

In the papers~\cite{GW1,GW2}, O.\;Guichard and A.\;Wienhard introduced a notion of $\Theta$-positivity for Lie groups that generalizes the notion of the Lusztig's total positivity for split real Lie groups~\cite{Lus94} to a larger class of semisimple Lie groups. The groups admitting such positive structures are of particular interest for higher rank Teichm\"uller theory because spaces of positive representations of the fundamental group of a closed surface $S$ into a semisimple Lie group $G$ with a positive structure provide examples of higher rank Teichm\"uller spaces: they provide subspaces of the character variety $\Rep(\pi_1(S),G)\coloneqq\Hom(\pi_1(S),G)/G$ that consist entirely of discrete and faithful representations.

Spaces of positive representations are not only interesting for closed surfaces but also for surfaces with punctures or boundary components. In fact, in these cases, the space of framed (and of decorated) representations admits particularly nice cluster structures. The space of framed representations of the fundamental group of a punctured surface of negative Euler characteristic into complex Lie groups and split real Lie groups was introduced by  V.~Fock and A.~Goncharov in~\cite{FG}. They study these spaces using ideal triangulations of surfaces, i.e.\ triangulations whose set of vertices agrees with the set of punctures of the surface. They introduce so-called $\mathcal X$-type and $\mathcal A$-type cluster coordinates associated with ideal triangulations and parametrize spaces of framed representations that are transverse to a given ideal triangulation. Moreover, they show that the positive locus of these coordinates is independent on the choice of triangulation and parametrizes higher rank Teichm\"uller spaces. In~\cite{BD14,BD17} F.~Bonahon and G.~Dreyer generalize this approach and give parametrizations of higher rank Teich\"uller spaces if $G=\PSL_n(\R)$ for closed surfaces using coordinates associated to geodesic laminations.

In the case of a split real Lie group $G$, the positive structure introduced by O.~Guichard and A.~Wienhard, agrees with Lusztig's total positivity if $\Theta$ is the set of all simple roots of $G$. In this case, when the surface is closed the space of positive representations agrees with the Hitchin components. The structure of Hitchin components is well-studied for closed surfaces~\cite{H92} as well as for surfaces with punctures and boundary components~\cite{FG}. In particular, their topology is well understood. More precisely, every Hitchin component is homeomorphic to an open ball in a Euclidean vector space. In particular, they are contractible manifolds.

Another example of Lie groups with a positive structure are Hermitian Lie groups of tube type. In this case positive representations are maximal representations which were introduced and studied in~\cite{BIW,BILW,Strubel}. The topology of spaces of maximal representation for closed surfaces was studied in~\cite{Gothen, Guichard_Wienhard_InvaMaxi, BGPG, AC}, partly using the theory of Higgs bundles. In~\cite{AGRW}, the spaces of framed and decorated maximal representations into the real symplectic group $\Sp(2n,\R)$ are parametrized using a noncommutative analog of the Fock--Goncharov parametrization~\cite{FG} and the topology of them is studied. In contrast to the Hitchin components, these spaces are not contractible, and their topology is more complicated.


In this article, we generalize the approach of~\cite{FG,AGRW} to a larger class of Lie groups. For a semisimple Lie group $G$ and a self-opposite parabolic subgroup $\Pp$ of $G$, we introduce the space of representations of the fundamental group of a punctured surface into $G$ framed by elements of the flag variety $G/\Pp$. We parametrize these spaces using parameters associated to ideal triangulations of the surface. If the group $G$ carries a positive structure in the sense of~\cite{GW2}, we parametrize spaces of positive representations as well. Using this parametrization, we study the topology of spaces of positive representations, its homotopy type and count the number of connected components.

We now describe our results in more detail.

Let~$S$ be a surface without boundary of negative Euler characteristic $\chi(S)$ and with $k>0$ punctures (we refer to Section~\ref{sec:top_data} for the wider generality that can be allowed for~$S$, for example disks with marked points on the boundary). Let $G$ be a semisimple Lie group and $\Pp$ be a self-opposite parabolic subgroup of $G$.

A \defin{framed representation} is a representation $\pi_1(S) \rightarrow G$ together with a $k$-tuples of flags $(F_1, \dots, F_k)$ of $\F\coloneqq G/\Pp$ that are fixed by the images of the peripheral elements $c_1, \dots, c_k$ in $\pi_1(S)$.

Fixing an ideal triangulation~$\mathcal{T}$ of $S$, we consider a subspace of the space of framed representations, that are transverse with respect to $\mathcal T$ (we refer to Definition~\ref{df:transverse_rep} for more details). Further, we introduce two collections of parameters: The first collection is associated to the triangles of $\mathcal T$, and is given by the elements of the unipotent subgroup $\Up$ of $\Pp$ that parametrize triples of flags  associated to the vertices of the triangles. There are $-2\chi(S)$ parameters in this first collection. To describe the second collection of parameters, we need to lift the ideal triangulation~$\mathcal{T}$ to the universal covering $\tilde S$ of $S$ and to fix a connected fundamental domain $D$ of $S$ that consists of triangles of the lifted triangulation. We associate to every pair of edges in the boundary of $D$ that are identified by an element of the fundamental group of $S$ an element of the Levi subgroup of $\L$ of $\Pp$. There are $1-\chi(S)$ parameters in this second collection. The following theorem shows that given these two collections of parameters up to a common conjugation by an element of $\L$, a framed representation can be uniquely reconstructed:

\begin{teo}\label{intro:main_thm}
The space of framed representations that are transverse with respect to $\mathcal T$ is homeomorphic to:
$$\bigl((\Up_*)^{-2\chi(S)}\times\L^{1-\chi(S)}\bigr)/\L$$
where $\Up_*$ is the open and dense subspace of $\Up$ that parametrizes pairwise transverse triples of flags in $\F$ and $\L$ acts by conjugation in every factor.
\end{teo}

If the group $G$ admits a positive structure, then we can similarly parameterize spaces of framed positive representations in the sense of~\cite{GW1,GLW}:
\begin{teo}\label{intro:main_thm_positive}
The space of positive framed representations is homeomorphic to:
$$\bigl((\Upp)^{-2\chi(S)}\times\L_0^{1-\chi(S)}\bigr)/\L_0,$$
where $\Upp$ is the positive semigroup of $\Up$ parametrizing positive triples of flags in $\F$ and $\L_0$ is the subgroup of the Levi subgroup $\L$ stabilizing $\Upp$ under the action by conjugation that acts by conjugation in every factor.
\end{teo}

As corollary of this result, we obtain a description of the homotopy type of the space of positive framed representations:

\begin{cor}
The space of positive framed representations is homotopy equivalent to $K_0^{1-\chi(S)}/K_0$, where $K_0$ is a maximal compact subgroup of $\L_0$.
\end{cor}

This immediately gives a count of the number of connected components of the space of positive framed representations. If $G$ is a connected Lie group, by the following Corollary, we also obtain a count of the connected components of the space of positive representations (without any framing), see Table~\ref{tab:conn_comp}.

\begin{cor}
The space of positive framed representations and the space of positive non-framed representations have the same number of connected components.
\end{cor}

\begin{table}[ht]
\begin{center}
\begin{tabular}{c|c}
Group & Number of connected components \\
\hline
Adjoint form of split real groups & 1 \\
\hline
Hermitian Lie groups & \\
$\Sp_{2n}(\R)$ & $2^{1-\chi(S)}$ \\
$\PSp_{2n}(\R)$, $n$ even & $2^{1-\chi(S)}$ \\
$\PSp_{2n}(\R)$, $n$ odd & 1 \\
$\UU(n,n)$ & 1 \\
$\SU(n,n)$ & 1 \\
$\SO^*(4n)$ & 1 \\
$E_{7(-25)}$ & 1\\
\hline
Orthogonal indefinite groups & \\
$\SO_0(1+p,n+p)$, $p$ odd & $2^{1-\chi(S)}$\\
$\SO_0(1+p,n+p)$, $p$ even & 1\\
\hline
Exceptional family & \\
$F_{4(4)}$ & 1\\
$E_{6(2)}$ & 1\\
$E_{7(-5)}$ & 1\\
$E_{8(-24)}$ & 1\\
\end{tabular}
\caption{Number of connected components of spaces of positive representations}\label{tab:conn_comp}
\end{center}
\end{table}
For exceptional Lie groups, we always consider the adjoint form. For further details about the topology of spaces of positive representations, we refer the reader to Section~\ref{sec:examples}.

\addtocontents{toc}{\SkipTocEntry}
\subsection*{Structure of the paper:}
In Section~\ref{sec:preliminaries} we recall some facts on the Lie theory and the notion of positivity for Lie group following~\cite{GW1,GW2,GLW}.
In Section~\ref{sec:top_data} we introduce punctured surfaces, ideal triangulations and some special kind of graphs on surfaces associated with triangulations that we are using later.
In Section~\ref{sec:framed_reps}, we introduce the spaces of framed representations and local systems on trees and describe the connection between them.
In Section~\ref{section:parametrization} we prove Theorem~\ref{intro:main_thm} on the parametrization of framed representations.
In Section~\ref{sec:positive} we introduce positive framed representations and describe several parametrizations of them, in particular, we prove Theorem~\ref{intro:main_thm_positive}.
In Section~\ref{sec:homotopy_type} we describe the homotopy type of the space of positive framed representations.
In Section~\ref{sec:connected_components} we prove that the number of connected components of the space of positive framed representations coincides with the number of connected components of the space of positive non-framed representations.
In Section~\ref{sec:examples} we apply the parametrizations from Sections~\ref{section:parametrization}~and~\ref{sec:positive} to understand explicitly the topology of spaces of transverse framed and positive framed representations for some Lie groups.

\section{Preliminary observations}\label{sec:preliminaries}

\subsection{Semisimple Lie groups and flag varieties}

Let $G$ be a connected semisimple Lie group with identity element $e$. Let $\Theta$ be a subset of the set of simple roots $\Delta$ of $\g\coloneqq \Lie(G)$. Let $\Pp$ (resp. $\Pm$) be the parabolic subgroup associated with $\Theta$.
We assume that $\Pp$ is self-opposite. By the Levi decomposition, $\Pp$ (resp. $\Pm$) is a semidirect product of its unipotent radical $\Up$ (resp. $\Um$) and the Levi factor $\L=\Pp\cap\Pm$.

Let $W$ be the Weyl group of $G$ and $w^0$ be the longest element of $W$. Let $\omega\in G$ be a lift of $w^0$. We fix $\omega$ once and for all.
Since $\Pp$ is self-opposite, $\omega\Pp\omega^{-1}=\Pm$ and $\omega^2\in \L$.
We consider the flag variety: $\F=G/\Pp$. The group $G$ acts on $\F$ transitively. We call elements of $\F$ \defin{flags}.

\begin{df}
Two flags $F_1,F_2\in\F$ are \defin{transverse} if there exists $g\in G$ such that $g(F_1,F_2)=(\Pp,\omega\Pp)$.
\end{df}

\begin{rem}\begin{itemize}
    \item The relation on $\F$ to be transverse is symmetric because $\omega(\Pp,\omega\Pp)=(\omega\Pp,\omega^2\Pp)=(\omega\Pp,\Pp)$.
    \item The stabilizer in $G$ of the pair $(\Pp,\omega\Pp)$ is equal to the Levi subgroup $\L$.
\end{itemize}
\end{rem}

\begin{prop}
For every flag $F\in\F$ which is transverse to $\Pp$ there exists a unique $u\in \Up$ such that $F=u\omega\Pp$.
\end{prop}

\begin{proof}
Since the pair $(\Pp,F)$ is transverse, there exists $g\in G$ such that $g(\Pp,F)=(\Pp,\omega\Pp)$. In particular, $g\in\Pp=\Stab_G(\Pp)$. Further, $F=g^{-1}\omega\Pp$. By the Levi decomposition, $g^{-1}=u\ell$ where $u\in\Up$, $\ell\in\L$. Since $\ell\omega\Pp=\omega\Pp$, we obtain $F=u\omega\Pp$.

Let $u'\in\Up$ be another element such that $F=u'\omega\Pp$. Therefore, $\omega\Pp=u^{-1}u'\omega\Pp$. This means that $u^{-1}u'\in\Pm\cap\Up=\{e\}$, i.e.\ $u=u'$.
\end{proof}

\subsection{Triples and quadruples of flags}
\begin{lem}\label{lem:triple}
Let $(F_1,F_2,F_3)$ be a triple of pairwise transverse flags. There exist an element $g\in G$ such that
$g(F_1)=\Pp$, $g(F_2)=\omega\Pp$, $g(F_3)=u\omega\Pp$ where $u\in\Up\cap\Pm\omega\Pm$.
\end{lem}

\begin{proof}
Indeed, up to $G$-action, we can assume $F_1=\Pp$, $F_2=\omega\Pp$, $F_3=u\omega\Pp$ for some $u\in\Up$. Since $F_2$ and $F_3$ are transverse, there exists $g\in G$ such that $gF_2=\Pp$, $gF_3=\omega\Pp$. From the first equality, we obtain $g=p\omega^{-1}$ for $p\in\Pp$. From the second one: $p\omega^{-1} u\omega\Pp=\omega\Pp$, i.e.\ $p\omega^{-1} u\in \Pm$. This means there exists $p'\in\Pm$ such that $p\omega^{-1} u=p'$, i.e.\ $u=\omega p p'=p''\omega p'$ where $p''=\omega p \omega^{-1}\in\Pm$. So $u\in\Pm\omega\Pm$.
\end{proof}

We denote $\Up_*\coloneqq\Up\cap \Pm\omega\Pm$. This is an open dense subset of $\Up$. The Levi factor $\L$ acts on $\Up_*$.

\begin{df}
We denote by $\Conf_3^*(\F)$ the space of transverse triples of flags in $\F$ up to the action of $G$.
\end{df}

The following Propositions are immediate:

\begin{prop}
The space $\Conf_3^*(\F)$ is homeomorphic to $\Up_*/\L$ where $\L$ acts by conjugation on $\Up_*$.
\end{prop}

\begin{prop}\label{prop:quadr}
Let $(F_1,F_2,F_3,F_4)$ be a quadruple of flags such that the triples $(F_1,F_3,F_4)$ and $(F_1,F_2,F_3)$ are transverse. Then there exists $g\in G$ such that
$$g(F_1,F_2,F_3,F_4)=(\Pp,\omega u' \omega\Pp,\omega\Pp, u\omega\Pp)=(\Pp,(u'')^{-1} \omega\Pp,\omega\Pp, u\omega\Pp)$$
where $u,u',u''\in\Up_*$. The element $g$ is unique up to the left multiplication by an element of $\L$, i.e.\ as an element of $\L\setminus G=\{\L g\mid g\in G\}$.
\end{prop}

The following proposition is quite technical, we will need it later to understand groups with positive structure. But since this proposition holds in general, we state it here.

\begin{prop}\label{uom_k}
Let $u\in\Up$ such that $(u\omega)^k\in\Pp$ for some $k\in\N$. Then $(u\omega)^k=(\omega u)^k\in\Norm_\L(u)$.
\end{prop}

\begin{proof}
Since $(\omega u)^k=u^{-1} (u\omega)^k u$, we obtain $(\omega u)^k\in\Pp$. Similarly $(u^{-1}\omega)^k=c^{-k}(\omega u)^{-k}\in\Pp$ and $(\omega u)^{-k}=c^{-k}(u \omega)^{-k}\in\Pp$ where $c=\omega^2\in Z(G)$.

Further, since $(u\omega)^k=\omega^{-1}(\omega u)^k\omega\in\omega^{-1}\Pp\omega=\Pm$. Similarly $(\omega u)^k, (\omega u^{-1})^k, (u^{-1}\omega)^k\in \Pm$. So all this four elements are in $\Pp\cap\Pm=\L$.

Now, $(u\omega)^{-k}u(u \omega)^k=(u\omega)^{-k+1}\omega^{-1} u^{-1} u (u\omega)^{k}= u(\omega u)^{-k}(u\omega)^{k}\in\Up$. Since $(\omega u)^{-k}(u\omega)^{k}\in \L$ and $\Pp$ is the semidirect product of $\Up$ and $\L$, we get $(\omega u)^{-k}(u\omega)^{k}=1$, i.e.\ $(u\omega)^{k}=(\omega u)^k$ and $(u\omega)^{-k}u(u \omega)^k=u$.
\end{proof}

\subsection{Groups with positive structure}\label{positive_groups}
In this section we follow the discussion from~\cite{GW1,GW2,GLW}.

Let $G$ be a semisimple Lie group with a $\Theta$-positive structure where $\Theta$ is a subset of the set of simple roots $\Delta$ of $\Lie(G)$.  The $\Theta$-positive structure on $G$ gives rise to a positive semigroup $\Upp$ of $\Up$. We denote by $\L_0$ the stabilizer of $\Upp$ in $\L$.

\begin{rem}
The subgroup $\L_0$ always contains the connected component of the identity element of $\L$ but not necessarily agree with it. In Section~\ref{sec:examples} we will see several examples of what $\L_0$ can be.
\end{rem}

Let $W$ be the Weyl group of $G$. It is generated by reflections $s_\alpha$ for all $\alpha\in\Delta$. By~\cite[Section 3.1]{GW2} there exits at most one special root $\beta_\Theta\in\Theta$. We define $\sigma_{\beta_\Theta}$ to be the longest element of the subgroup $W_{\{\beta_\Theta\}\cup(\Delta\bs\Theta)}\leq W$ generated by $s_\alpha$ for all $\alpha\in\{\beta_\Theta\}\cup(\Delta\bs\Theta)$, and $\sigma_{\beta}\coloneqq s_\beta$ for all $\beta\in\Theta\bs\{\beta_\Theta\}$. The group $W(\Theta)$ is defined to be the subgroup $W$ generated by $\sigma_\beta$ for all $\beta\in\Theta$.

Let $\Sigma_\Theta^+\coloneqq \Sigma^+\bs\Span(\Delta\bs\Theta)$ (resp. $\Sigma_\Theta^-\coloneqq -\Sigma_\Theta^+$), where $\Sigma^+$ is the set of positive roots of $\Lie(G)$. For every $\beta\in\Theta$ we consider
$$\up_\beta\coloneqq\bigoplus_{\tiny\begin{matrix}\alpha\in\Sigma^+_\Theta \\ \beta-\alpha\in\Span(\Delta\bs\Theta)\end{matrix}} \g_\alpha$$
where $\g_\alpha\subseteq \g$ is the root space corresponding to the root $\alpha$. there exists an $\L_0$-invariant proper convex cone $c_\beta\subset \up_\beta$ of full dimension (i.e.\ $c_\beta$ does not contain lines and it generates $\up_\beta$ as a vector space). We denote $\mathring c_\beta$ the interior of $c_\beta$.

There exists a map (see \cite[Theorem~1.3]{GW2}):
$$\begin{matrix}
F\colon & \up_{\beta_{i_1}}\times\dots\times \up_{\beta_{i_l}} & \to & \Up\\
& (v_{i_1},\dots,v_{i_l}) & \mapsto & \prod_{j=1}^l\exp(v_{i_j}).
\end{matrix}$$
such that the restriction $F|_{\mathring c_{\beta_{i_1}}\times\dots\times \mathring c_{\beta_{i_l}}}$ is a diffeomorphism onto $\Upp$, i.e.\ every element $u\in \Upp$ can be written in a unique way as $u=u(v_{i_1},\dots,v_{i_l})=\prod_{j=1}^l\exp(v_{i_j})$ where $v_{i_j}\in \mathring c_{\beta_{i_j}}$. Moreover, $F$ is $\L$-equivariant

\begin{prop}[{\cite[Proposition~3.7]{GW2}}]\label{Levi-symmetric_space}
The group $\L_0$ acts on $\prod_{\beta\in\Theta} \mathring c_{\beta}$ componentwise by the adjoint action. This action is transitive, and the stabilizer of every point is a maximal compact subgroup of $\L_0$. In~particular, $\prod_{\beta\in\Theta} \mathring c_{\beta}$ is a model of the symmetric space of $\L_0$.
\end{prop}

We choose $v_\beta\in \mathring c_\beta$ for every $\beta\in\Theta$, and take the element
\begin{equation}\label{special_u}
u_\Theta\coloneqq F(v_{\beta_{i_1}},\dots,v_{\beta_{i_l}}).
\end{equation}
The stabilizer of $u_\Theta$ in $\L_0$ is a maximal compact subgroup of $\L_0$. By~\cite[Section~2.5]{GLW}, there is a group $H_\Theta$ that is isogenic to $\PSL_2(\R)$ containing $u_\Theta$. For $H_\Theta$ there exists a split real subgroup $H'_\Theta$ of $G$ with the Weyl group $W(\Theta)$ and such that $H_\Theta$ is the principal $\sl_2$-subgroup of $H'_\Theta$. Let $w_\Theta^0$ be the nontrivial element in the Weyl group of $H_\Theta$. Then $w_\Theta^0$ is the longest element in $W(\Theta)$. We denote by $\omega_\Theta$ a lift of $w_\Theta^0$ in $H_\Theta$.

\begin{prop}\label{properties_special_u}
$(u_\Theta\omega_\Theta)^3=(\omega_\Theta u_\Theta)^3\in\Norm_{\L_0}(u_\Theta)$.
\end{prop}

\begin{proof}
Without loss of generality, assume $G$ to be the adjoint form. By~\cite[Section~11.2]{GW2}, there exist $q\colon\SL_2(\R)\to H_\Theta$ a positive principal homomorphism which maps upper triangular matrices to $\Pp$, lower triangular matrices to $\Pm$ and such that $q\begin{pmatrix}1 & 1 \\ 0 & 1\end{pmatrix}=u_\Theta$ and $q\begin{pmatrix}0 & 1 \\ -1 & 0\end{pmatrix}=\omega_\Theta$.

An easy computation shows that $(u_\Theta\omega_\Theta)^3$ is central in $H_\Theta$, in particular it is in $\Pp$. Applying Proposition~\ref{uom_k} with $k=3$, we get $(u_\Theta\omega_\Theta)^3=(\omega_\Theta u_\Theta)^3\in\Norm_{\L}(u_\Theta)$. Since $\Upp$ is connected, the conjugation by $(u_\Theta\omega_\Theta)^3$ fixed $\Upp$, i.e.\ $(u_\Theta\omega_\Theta)^3\in\Norm_{\L_0}(u_\Theta)$.
\end{proof}

Let $\Phi\colon G\to G$ be a group automorphism that preserves $\Up$. Slightly abusing the notation, we will also write $\Phi\colon \Up\to\Up$. Following~\cite{GLW}, we call $\Phi(\Upp)\omega\Pp$ a~\defin{standard diamond}. For $\Phi=\Id$, the diamond $\mathcal D^+\coloneqq\Upp\omega\Pp$ is called the \defin{standard positive diamond}. The diamond $\mathcal D^-\coloneqq (\Upp)^{-1}\omega\Pp$ that is opposite to $\mathcal D^+$ is called the \defin{standard negative diamond}. This is a standard diamond because $(\Upp)^{-1}=\Phi(\Upp)$ where $\Phi$ is the well defined automorphism of $G$ whose derivative $\phi$ at the level of Lie algebras acts in the following way: $\phi(x)=-x$ for all $x\in\up_\beta$ for $\pm\beta\in\Theta$, $\phi(x)=x$ for all $x\in\up_\beta$ for $\pm\beta\in\Delta\bs\Theta$ and all $x$ in the Levi factor.

By~\cite[Proposition 10.1]{GW2}, standard diamonds are connected component of $\Omega^+\cap\Omega^-$ where $\Omega^+\subset \F$ is the spaces of all flags transverse to $\Pp$ and  $\Omega^-\subset \F$ is the spaces of all flags transverse to $\Pm$.

\begin{prop}
The element $\omega$ exchanges the standard negative diamond and the standard positive one, i.e.\ $\omega\mathcal D^-=\mathcal D^+$ and $\omega\mathcal D^+=\mathcal D^-$. In particular, for every $u\in(\Upp)^{-1}$ there exists unique $u'\in\Upp$ such that $\omega u=u'p$ for some $p\in\Pm$.
\end{prop}

\begin{proof} First notice that $\omega$ permutes $\Pp$ and $\Pm$, so it acts by permutation on standard diamonds.

Further, notice that this statement holds for $\PSL(2,\R)$. Because the center of $G$ acts trivially on $\F$, without loss of generality, we assume that $G$ is the adjoint form. We take a positive embedding $q\colon\SL(2,\R) \to H_\Theta\leq G$. Under the map $q$, the upper triangular matrices are mapped to $\Pp$ and unipotent matrices with positive entry in the position $(1,2)$ are mapped to $\Upp$. The map $q$ induces the injective homomorphism at the level of Weyl groups that maps a generator of the Weyl group of $\SL(2,\R)$ to the longest element $w_\Theta^0$ of the group $W(\Theta)$ of $G$. In particular, $\omega_\Theta\coloneqq q\begin{pmatrix}0 & 1 \\ -1 & 0\end{pmatrix}$ is a lift of the longest element $w^0_\Theta\in W(\Theta)$.

Let $C$ be the flag variety for $\SL(2,\R)$ which is the real projective line. We denote by $C_q\subseteq\F$ the positive circle corresponding to the embedding $q$, i.e.\ the orbit of $\Pp$ in $\F$ under the action of $q(\PSL(2,\R))$. The map $q$ induces the $q$-equivariant embedding $q'\colon C\to C_q\subset \F$. Notice that $q'([1:0])=\Pp$, $q'([0:1])=\omega\Pp$, $q'([a:1])\in\mathcal D^+$ for all $a>0$ and $q'([-a:1])\in\mathcal D^-$ for all $a>0$.

Since $(w^0)^{-1}w^0_\Theta$ is an element of the subgroup group $W_{\Delta\bs\Theta}\subseteq W$ (see~\cite[Section~4]{GW2}), the elements $\omega_\Theta$ and $\omega$ act in the same way on $C_q$. Finally we obtain:
$$\omega(q'([1:1]))=\omega_\Theta(q'([1:1]))=q'\left(\begin{pmatrix}0 & 1 \\ -1 & 0\end{pmatrix}[1:1]\right)=q'([-1:1])\in \mathcal D^-.$$
Since $\mathcal D^+$ and $\mathcal D^-$ are connected components of $\Omega^+\cap\Omega^-$, for every $F\in\mathcal D^+$, $\omega F\in \mathcal D^-$.
\end{proof}

\begin{df}
A $k$-tuple $(F_1,\dots,F_k)$ ($k\geq 3$) of flags in $\F$ is called \defin{positive} if there exists a $g\in G$ such that $g(F_1,\dots,F_k)=(\Pp, \omega\Pp, u_1\omega\Pp, u_1u_2\omega\Pp,\dots, u_1,\dots u_{k-2}\omega\Pp)$ where $u_1,\dots,u_{k-2}\in\Upp$.
\end{df}

\begin{fact}[Section~2.4 of \cite{GLW}]\label{properties_positive_tuple}
The element $g$ in the definition above is not unique, but all of them are related by the conjugation be an element of $\L_0$. Moreover, if $(F_1,\dots,F_k)$ is positive and there exists a $g$ such that $g(F_1,F_2,F_3)=(\Pp, \omega\Pp, u_1\omega\Pp)$ for $u_1\in\Upp$, then there exist $u_2,\dots,u_{k-2}\in\Upp$ such that $g(F_1,\dots,F_k)=(\Pp, \omega\Pp, u_1\omega\Pp, u_1u_2\omega\Pp,\dots, u_1,\dots u_{k-2}\omega\Pp)$.

A $k$-tuple is positive if and only if
\begin{itemize}
\item every of its cyclically ordered quadruple is positive.
\item every of its cyclic permutation is positive.
\end{itemize}
\end{fact}

\begin{df}
We denote by $\Conf_k^+(\F)$ the space of positive $k$-tuples of flags in $\F$ up to the action of $G$.
\end{df}

The following Propositions are immediate:

\begin{prop}
The space $\Conf_k^+(\F)$ is homeomorphic to $(\Upp)^{k-2}/\L_0$.
\end{prop}

\begin{prop}
Let $(F_1,F_4,F_2,F_3)$ be a positive quadruple of flags. Then there exists $g\in G$ such that
$$g(F_1,F_4,F_2,F_3)=(\Pp,\omega u' \omega\Pp,\omega\Pp, u\omega\Pp)=(\Pp,(u'')^{-1}\omega\Pp,\omega\Pp, u\omega\Pp)$$
where $u,u',u''\in\Upp$. The element $g$ is unique up to a left multiplication by an element of $\L_0$, i.e.\ as an element of $\L_0\setminus G=\{\L_0 g\mid g\in G\}$.
\end{prop}


\section{Topological data}\label{sec:top_data}

\subsection{Punctured surfaces}
Let $\bar S$ be a compact oriented smooth surface of finite type with or without boundary. Let $P$ be a nonempty finite subset of $\bar S$ such that on every boundary component there is at least one element of $P$. We define $S\coloneqq\bar S\bs P$. We assume that the Euler characteristic $\chi(S)$ of $S$ is negative or that $\bar S$ is diffeomorphic to a closed disc and there are at least three punctures. Surfaces that can be obtained in this way are called \defin{punctured surfaces} (some authors call them also ciliated surfaces, e.g.~\cite{palesi}). Elements of $P$ are called \defin{punctures} of $S$. Sometimes we will distinguish between elements of $P$ that lie in the interior of $\bar S$ -- \defin{internal punctures} and that lie on the boundary -- \defin{external punctures}.

Every punctured surface can be equipped with a complete hyperbolic structure of finite volume with geodesic boundary. For every such hyperbolic structure all the punctures are cusps and all boundary curves are (infinite) geodesics. Once equipped with a hyperbolic structure as above, the universal covering $\tilde S$ of $S$ can be seen as closed subset of the hyperbolic plane $\HH^2$ which is invariant under the natural action of $\pi_1(S)$ on $\HH^2$ by the holonomy representation. Punctures $P$ of $S$ are lifted to points of the ideal boundary of $\HH^2$ which we call punctures of $\tilde S$ and denote by $\tilde P\subset \partial_\infty \tilde S\subseteq \partial_\infty\HH^2$. Notice that if $\bar S$ has no boundary, then $\tilde S$ is the entire $\HH^2$.

\subsection{Ideal triangulations}
An \defin{ideal triangulation} of $S$ is a triangulation of $\bar S$ whose set of vertices agrees with $P$. We always consider edges of an ideal triangulation as homotopy classes of (non-oriented) paths connecting points in $P$ (relative to its endpoints). Connected components of the complement of all edges of an ideal triangulation $\mathcal T$ in $S$ are called faces or triangles of $\mathcal T$. Every edge belongs to the boundary of one or two triangles. In the first case, an edge is called \defin{external}, in the second -- \defin{internal}. We denote by $E$, resp. $E_{in}$, resp. $E_{ex}$ the set of all edges, resp. all internal edges, resp. all external edges of $\T$. We denote by $T$ the set of faces of $\T$ and fix a total order $<$ on $T$. Any ideal triangulation of $S$ can be realized as an ideal geodesic triangulation as soon as a hyperbolic structure as above on $S$ is chosen. We fix a hyperbolic structure of finite volume on $S$ and assume that edges of $\T$ are geodesics.

The following lemma counts the number of triangles of an ideal triangulation of $S$:

\begin{lem}
Let $S$ be a surface of genus $g$ with $p_i$ internal punctures and $p_e$ external punctures. Let $m$ be the number of boundary components of $\bar S$, then
$$\#T=4g-4+2p_i+2m+p_e=p_e-2\chi(S).$$
where $\chi(S)=2-2g-p_i-m$ is the Euler characteristic of $S$.
\end{lem}

\subsection{Graph \texorpdfstring{$\Gamma$}{Gamma}}\label{sec:Graph_Gamma}
\begin{figure}[ht]
\begin{tikzpicture}[every node/.style={fill, circle, inner sep = 1.5pt}]
\draw (-5,0) -- (1,4.5) -- (7,0.5) -- (1,-3.5) -- (-5,0);
\draw (1,4.5) -- (1,-3.5);

\node[white, label=$\tau$] at (2,3) {};
\node[white, label=$\tau'$] at (0.2,3) {};

\node[label=above left:$v_\tau$](vt) at (1.8,0.5) {};
\node[label=above right:$v_{\tau'}$](vt') at (-2+0.3,2-0.3) {};
\node[label=above right:$v_2$](v2) at  (3.5-0.2,-1.5+0.2) {};
\node[label=below:$v_3$](v3) at (0.2,0.5) {};
\node[label=right:$v_1$](v4) at (-2,-1) {};

\draw[middlearrow={latex}] (vt) to (v3) ;


\draw[middlearrow={latex}] (vt) to (v2) ;

\draw[middlearrow={latex}] (vt') to (v3) ;

\draw[middlearrow={latex}] (vt') to (v4) ;

\node (v5) at (-2.5,2.5) {};
\node (v7) at (4,-2) {};
\node (v8) at (-2.2,-2.2) {};

\draw[middlearrow={latex}] (vt') to (v5) ;


\draw[middlearrow={latex}] (v2) to (v7) ;

\draw[middlearrow={latex}] (v8) to (v4) ;

\end{tikzpicture}
\caption{Example of a graph $\Gamma$. Here the upper right edge of the triangulation is external, all other edges are internal, $v_\tau,v_{\tau'}\in V^0_\Gamma$, $v_1,v_2,v_3\in V_\Gamma\setminus V^0_\Gamma$, $\tau<\tau'$.}
\end{figure}
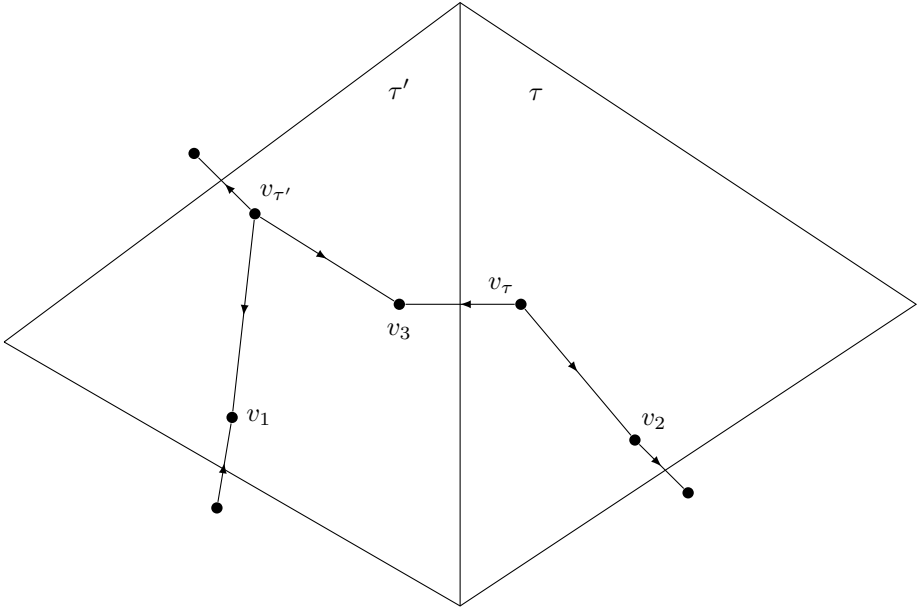

We construct the following graph $\Gamma$ on $S$. The set of vertices of $\Gamma$ is $V_\Gamma\coloneqq\{(\tau,e)\in T\times E_{in}\mid e\subset\bar \tau\}$. A vertex $(\tau,e)\in V_\Gamma$ can be seen as a point in the triangle $\tau\in T$ lying close to the internal edge $e\in E_{in}$.  In every triangle $\tau\in T$ we choose one vertex $v_\tau\in V_\Gamma$ in the triangle $\tau$. We denote: $V^0_\Gamma\coloneqq\{v_\tau\mid \tau\in T\}$.

We now describe the set of oriented edges $E^+_\Gamma$ of $\Gamma$. First we fix the following notation: the oriented edge from $v\in V_\Gamma$ to $v'\in V_\Gamma$ is denoted as $(v\to v')$ and as $(v'\gets v)$.
\begin{itemize}
\item Let vertices $v$ and $v'$ lie in the same triangle and $v\in V^0_\Gamma$ or $v'\in V_\Gamma$, then $(v\to v')$ is an edge of $\Gamma$.
\item Let $v=(\tau,e)$, $v'=(\tau',e)$ and $\tau<\tau'$, then $(v \to v')$ is an edge of $\Gamma$. This edge crosses the internal edge $e$ of the triangulation $\mathcal T$.
\end{itemize}

\begin{figure}[ht]
\begin{tikzpicture}

\draw (-3,3) node[label = $v^t$] (v1) {} -- (-3,-2) node[label = below:$v^b$] (v2) {} -- (0.7,0.5) node[label = right:$v^r$] (v3) {} -- (-3,3);
\draw [dotted](-3,3) -- (-2.2,0.5);
\draw [dotted](-3,-2) -- (-2.2,0.5);
\node[fill, circle, inner sep = 1.5pt, label= right:$v$] at (-2.2,0.5) {};
\node[label= $e$] at (-3.2,0.2) {};
\node[label= $\tau$] at (-1.2,0.8) {};
\end{tikzpicture}
\caption{Definition of $v^t$, $v^b$ and $v^r$ for $v\in V_\Gamma$}\label{vt-vb}
\end{figure}
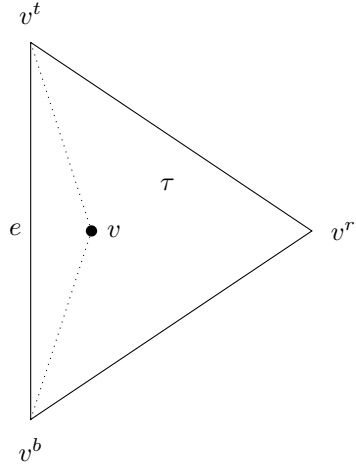

Let $v=(\tau,e)\in V_\Gamma$ and the edge $r$ connects two punctures $p,p'\in P$. If we connect $v$ with $p$ and $p'$ by two simple non intersecting segments in $\tau$, we obtain a triangle with vertices $v,p,p'$ in $\tau$. We denote $v^t\coloneqq p$, $v^b\coloneqq p'$ if the orientation of the triangle $(v,p,p')$ agrees with the orientation of the surface (see Figure~\ref{vt-vb}). Further we define $v^r$ to be the unique puncture of $P$ such that $(v^t,v^b,v^r)=\tau$.

We also denote $E^-_\Gamma\coloneqq \{(v'\gets v)\mid (v\gets v')\in E^+_\Gamma\}$. Since $\Gamma$ is a tree, $E^+_\Gamma\cap E^-_\Gamma=\emptyset$. We denote $E_\Gamma\coloneqq E^+_\Gamma\cup E^-_\Gamma$.

We denote by $\tilde\Gamma$ the lift of $\Gamma$ to the universal covering $\tilde S\subset \HH^2$ of $S$. We also choose a point $b\in V_\Gamma$ and fix one of its lifts $\tilde b\in V_{\tilde\Gamma}$. Assume $b\in \tau_0$ and $\tilde b\in \tilde \tau_0$ for $\tau_0\in T$ and $\tilde \tau_0\in \tilde T$.

\begin{rem}
The graph $\tilde\Gamma$ is a tree. The graph $\Gamma$ is homotopy equivalent to $S$.
\end{rem}

\section{Framed representations and framed local systems}\label{sec:framed_reps}
\subsection{Transverse framed representations}

\begin{df}
The space of all homomorphisms $\rho\colon\pi_1(S)\to G$ is denoted by $\Hom(\pi_1(S),G)$. The group $G$ acts on $\Hom(\pi_1(S),G)$ by conjugation. The quotient space is denoted by $\Rep(\pi_1(S),G)\coloneqq\Hom(\pi_1(S),G)/G$.
\end{df}

\begin{df}
A \defin{framing} is a map $F\colon \tilde P\to \F$ where $\tilde P$ is the set of lifted punctures of $S$ to $\tilde S$ as before. Let $\rho\colon\pi_1(S)\to G$ be a homomorphism. A \defin{framing} of $\rho$ is a $\pi_1(S)$-equivariant map $F\colon \tilde P\to \F$, i.e.\ for every $\gamma\in\pi_1(S)$, $F(\gamma(\tilde p))=\rho(\gamma)F(\tilde p)$ for all $\tilde p\in \tilde P$. A \defin{framed homomorphism} is a pair $(\rho,F)$ where $F$ is a framing of $\rho$.

The space of all framed homomorphisms is denoted by $\Hom^{\fr}(\pi_1(S),G)$. The group $G$ acts by conjugation on the space of homomorphisms. It also acts on framings as follows: for a framing $F$ and $g\in G$ the map $g(F)\colon\tilde P\to \F$ defined by $g(F)(p)=g(F(p))$ is again a framing. Moreover, if $(\rho,F)$ is a framed homomorphism, then $(g\rho g^{-1},g(F))$ is again a framed homomorphism. So we can consider the following quotient space:
$$\Rep^{\fr}(\pi_1(S),G)\coloneqq \Hom^{\fr}(\pi_1(S),G)/G$$
which is called the \defin{space of framed representations}. A \defin{framed representation} is an element of $\Rep^{\fr}(\pi_1(S),G)$.
\end{df}

Let $\T$ be an ideal triangulation of $S$ and $\tilde\T$ is the lift of $\T$ to the universal covering $\tilde S$.

\begin{df}\label{df:transverse_rep}
A framed homomorphism $(\rho,F)$ is called \defin{$\T$-transverse} if for every two punctures $\tilde p_1,\tilde p_2\in\tilde P$ that are connected by an edge of $\tilde\T$, the flags $F(\tilde p_1)$ and $F(\tilde p_2)$ are transverse.

The space of all $\T$-transverse framed homomorphisms is denoted by $\Hom^{\fr}_\T(\pi_1(S),G)$. The space
$$\Rep^{\fr}_\T(\pi_1(S),G)\coloneqq \Hom^{\fr}_\T(\pi_1(S),G)/G$$
is called the \defin{space of $\T$-transverse framed representations}.
\end{df}

\subsection{Framed local systems on trees}
Let $\Gamma$ be an oriented tree. 
We denote by $V_\Gamma$ the set of vertices of $\Gamma$ and by $E^+_\Gamma$ the set of oriented edges. If $e\in E^+_\Gamma$ is an edge going from the vertex $v\in V_\Gamma$ to the vertex $v'\in V_\Gamma$, we write $e=(v\to v')=(v'\gets v)$. We also denote $E^-_\Gamma\coloneqq \{(v'\gets v)\mid (v\gets v')\in E^+_\Gamma\}$. Since $\Gamma$ is a tree, $E^+_\Gamma\cap E^-_\Gamma=\emptyset$. We denote $E_\Gamma\coloneqq E^+_\Gamma\cup E^-_\Gamma$.

\begin{df}
A \defin{transverse framing} of $\Gamma$ is a pair of maps $F^t,F^b\colon V_\Gamma\to \F$ such that $F^t(v)$ and $F^b(v)$ are transverse for every $v\in V_\Gamma$.
\end{df}

\begin{df} A \defin{$G$-local system on a tree $\Gamma$} is a map $T\colon V_\Gamma\cup E_\Gamma\to G$ such that for any edge $(v'\gets v)\in E_\Gamma$, $v,v'\in V_\Gamma$, $T(v'\gets v)=T(v\gets v')^{-1}$ and  $T(v')=T(v'\gets v)T(v)$. A $G$-local system on $\Gamma$ defined by the map $T$ is denoted by $(\Gamma,T)$. 

Let $(F^t,F^b)$ be a transverse framing, let $(\Gamma,T)$ be a $G$-local system on $\Gamma$. We say that the transverse framing $(F^t,F^b)$ is \defin{adapted} to the local system $(\Gamma,T)$ if for every $v\in V_\Gamma$ holds: $T(v)F^t(v)=\Pp$, $T(v)F^b(v)=\omega\Pp$. A quadruple $(\Gamma,T,F^t,F^b)$ is called a \defin{transverse framed $G$-local system} if $(F^t,F^b)$ is a transverse framing adapted to $(\Gamma,T)$. 
\end{df}

\begin{rem}
Let $v\in V_\Gamma$ be a fixed vertex of $\Gamma$. The map $T$ is determined by $T|_{E^+_\Gamma}$ and $T(v)$. The elements $T(v)$ and $T(e)$ for all $e\in E^+_{\Gamma}$ can be arbitrary elements of $G$.
\end{rem}

Any path $\gamma$ in $\Gamma$ can be written in a unique way as a sequence $\gamma=(v_k\gets v_{k-1}\gets\dots\gets v_1)$ where $v_1$ is the starting vertex of $\gamma$ and $v_k$ is the end-vertex of $\gamma$ and $(v_{i+1}\gets v_{i})\in E_\Gamma$ for all $i\in\{1,\dots,k-1\}$. Sometimes to simplify the notation, we will just write $\gamma=(v_k\gets v_1)$. We also extend the map $T$ to the space of paths as follows: if $\gamma=(v_k\gets v_{k-1}\gets\dots\gets v_1)$, then $T(\gamma)\coloneqq T(v_k\gets v_{k-1})\dots T(v_2\gets v_1)$.

\subsection{From local systems to representations}\label{Loc_to_Rep}

Let $\mathfrak L=(\Gamma,T,F^t,F^b)$ be a transverse framed $G$-local system on a tree $\Gamma$. Let $H$ be a subgroup of the group $\Aut(\Gamma)$ of all automorphisms of $\Gamma$.

\begin{df}
The local system $\mathfrak L=(\Gamma,T,F^t,F^b)$ is called \defin{$H$-invariant} if for all $\gamma\in H$ and for all edges $(w\gets v)\in E_\Gamma$, $v,w\in V_\Gamma$, $T(w\gets v)=T(\gamma w\gets \gamma v)$.
\end{df}

An $H$-invariant $G$-local system always gives rise to a homomorphism $\rho\colon H\to G$ in the following way: let $v\in V_\Gamma$, we define for every $\gamma \in H$:
\begin{equation}\label{Graph_representation}
\rho_\mathfrak L(\gamma)\coloneqq T(\gamma v)^{-1}T(v).
\end{equation}

\begin{rem}
If $\mathfrak L=(\Gamma,T,F^t,F^b)$ is an $H$-invariant local system, then the framing $(F^t,F^b)$ is $\rho_\mathfrak L$-equivariant, i.e.\ for all $v\in V_\Gamma$, $F_t(\gamma v)=\rho_\mathfrak L(\gamma)(F_t(v))$ and $F_b(\gamma v)=\rho_\mathfrak L(\gamma)(F_b(v))$.
\end{rem}

\begin{prop}
The map $\rho_\mathfrak L$ is a group homomorphism that does not depend on the choice of $v\in V_\Gamma$.
\end{prop}

\begin{proof}
Let $w\in V_\Gamma$ be another vertex. Then
$$T(\gamma w)^{-1}T(w)
=(T(\gamma w\gets \gamma v)T(\gamma v))^{-1}(T(w\gets v)T(v))$$
$$=T(\gamma v)^{-1}T(\gamma w\gets \gamma v)^{-1}T(w\gets v)T(v)
=T(\gamma v)^{-1}T(v)$$
since $T(\gamma w\gets \gamma v)=T(w\gets v)$ by $H$-invariance. So $\rho_\mathfrak L$ does not depend on the choice of $v\in V_\Gamma$.

Let now $\gamma_1,\gamma_2\in H$. We obtain $\rho_\mathfrak L(\gamma_2)=T(\gamma_2 (\gamma_1 v))^{-1}T(\gamma_1 v)$ and $\rho_\mathfrak L(\gamma_1)=T(\gamma_1 v)^{-1}T(v)$. Therefore,
$$\rho_\mathfrak L(\gamma_2)\rho_\mathfrak L(\gamma_1)=T(\gamma_2 (\gamma_1 v))^{-1}T(\gamma_1 v)T(\gamma_1 v)^{-1}T(v)=T(\gamma_2 (\gamma_1 v))^{-1}T(v)=\rho_\mathfrak L(\gamma_2\gamma_1).$$
So $\rho_\mathfrak L$ is a group homomorphism.
\end{proof}



Let $S$ be a punctured surface as before. Let $\mathcal T$ be an ideal triangulation of $S$ and $\tilde\Gamma$ be a graph on $\tilde S$ defined in Section~\ref{sec:Graph_Gamma}. Every framing $F\colon \tilde P\to\F$ defines the unique framing $(F^t,F^b)$ on $\tilde\Gamma$ as follows: for every vertex $v$ of $\tilde\Gamma$ define $F^t(v)\coloneqq F(v^t)$ and $F^b(v)\coloneqq F(v^b)$ (see Figure~\ref{vt-vb}). In this case we say that $(F^t,F^b)$ is \defin{induced} by $F$. We also denote $F^r(v)\coloneqq F(v^r)$.

Let $r$ be an internal edge of $\tilde\Gamma$ that separates two triangles $\tau$ and $\tau'$, and $v\coloneqq (\tau,r)$, $v'\coloneqq (\tau',r)$. Then $F^t(v')=F^b(v)$ and $F^b(v')=F^t(v)$. Further, let $v$ and $v'$ two vertices of $\tilde\Gamma$ that lie in the same triangle and $(v')^t=v^b$. Then $F^t(v')=F^b(v)$, $F^b(v')=F^r(v)$, $F^r(v')=F^t(v)$.

The following proposition is immediate:

\begin{prop}
Let $(F^t,F^b)$ be a transverse framing on $\tilde\Gamma$ such that
\begin{itemize}
    \item for every internal edge $r$ of $\tilde\Gamma$ that separates two triangles $\tau$ and $\tau'$, and $v\coloneqq (\tau,r)$, $v'\coloneqq (\tau',r)$ holds $F^t(v')=F^b(v)$;
    \item for every two vertices $v$ and $v'$ of $\tilde\Gamma$ that lie in the same triangle and $(v')^t=v^b$ holds $F^t(v')=F^b(v)$.
\end{itemize}
Then there exist a unique framing $F\colon \tilde P\to\F$ that induces $(F^t,F^b)$.
\end{prop}



\subsection{From representations to local systems}\label{Moves}

In this section, we show how to construct a transverse framed local system on $\Gamma$ out of a transverse framed representation.

Let $\mathcal T$ be an ideal triangulation of $S$ and $\tilde\Gamma$ be a graph on $\tilde S$ defined in Section~\ref{sec:Graph_Gamma}. Given a framed homomorphism $(\rho,F)\in\Hom^{\fr}_\T(\pi_1(S,b),G)$, then the framing $F$ gives rise to the $\rho$-equivariant transverse framing $(F^t,F^b)$ of $\tilde\Gamma$.

\begin{rem}
If $\mathfrak L=(\tilde\Gamma,T,F^t,F^b)$ is a $\pi_1(S)$-invariant transverse framed trivial $G$-local system on $\tilde\Gamma$, then $F$ is a framing of $\rho_\mathfrak L$, but in general, $\rho\neq \rho_\mathfrak L$.
\end{rem}

We are going to define a local system $\mathfrak L=(\tilde\Gamma,T)$ such that $(F^t,F^b)$ will be its transverse framing and  $\rho_\mathfrak L=\rho$. We start with the two simplest cases when $S$ is a triangle and a triangulated quadrilateral, and then we describe a procedure for a general surface $S$.

\subsubsection{Triangle. Turn to the right and turn to the left in a triangle}\label{Turn}

Let $S=\tilde S$ be an ideal triangle with vertices $p_1,p_2,p_3$. The only triangulation $\T$ of $S$ consists of all external edges of $S$. Let $(\rho,F)$ be a $\T$-transverse framed homomorphism. Since $\pi_1(S)$ is trivial, $\rho$ is trivial. Let $F(p_i)=F_i$ for all $i\in\{1,2,3\}$. Let the graph $\Gamma=\tilde\Gamma$ consist of three vertices $v$, $v'$ and $v''$ and two edges $(v'\gets v)$ and $(v''\gets v)$.

Because of transversality by Lemma~\ref{lem:triple}, there exists an element $g\in G$ such that
$$g(F_1,F_2,F_3)=(\Pp,\omega\Pp,u\omega\Pp)$$
with $u\in\Up_*$ (see Figure~\ref{Turn_figure}).

We define $T(v)\coloneqq g$ and $T(v'\gets v)\coloneqq \omega u^{-1}$, then $T(v')=\omega u^{-1}g$ and
$$T(v')F^t(v')=T(v')F_3=T(v'\gets v)u\omega\Pp=\Pp,$$
$$T(v')F^b(v')=T(v')F_1=T(v'\gets v)\Pp=\omega \Pp.$$
Moreover $T(v')F_2=T(v'\gets v)\omega\Pp=\omega u^{-1}\omega \Pp$. Because of transversality there exist unique $\tilde u\in\Up_*$ and $q\in \Pp$ such that $\omega u^{-1}\omega=\tilde u\omega q$. This means:
$$T(v')F_2=T(v'\gets v)\omega\Pp=\tilde u\omega \Pp.$$

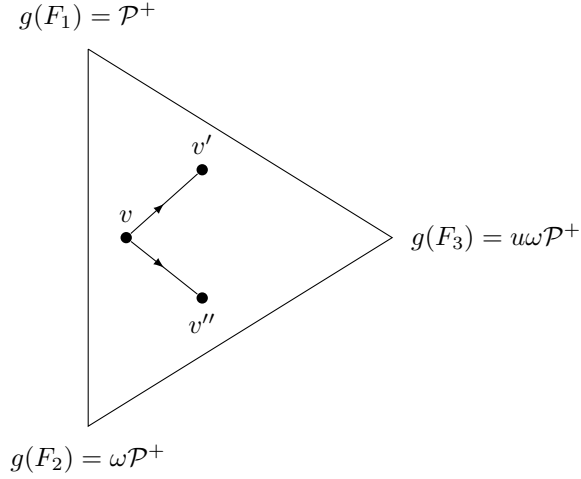
\begin{figure}[ht]
\begin{tikzpicture}
\draw (-3,3) node[label = ${g(F_1)=\Pp}$] (v1) {} -- (-3,-2) node[label = below:${g(F_2)=\omega\Pp}$] (v2) {} -- (1,0.5) node[label = right:${g(F_3)=u\omega\Pp}$] (v3) {} -- (-3,3);
\node[fill, circle, inner sep = 1.5pt, label=$v$](a) at (-2.5,0.5) {};
\node[fill, circle, inner sep = 1.5pt, label=$v'$](b) at (-1.5,1.4) {};
\node[fill, circle, inner sep = 1.5pt, label= below:$v''$](b') at (-1.5,-0.3) {};

\draw[middlearrow={latex}] (a) to (b) ;
\draw[middlearrow={latex}] (a) to (b');

\end{tikzpicture}
\caption{Turn to the left and turn to the right. Picture to have in mind}\label{Turn_figure}
\end{figure}

We call the transformation $T(v'\gets v)=:T_l(u)$ ``turn to the left in the triangle $(\Pp,\omega\Pp,u\omega\Pp)$''. Notice that the matrices $u$ and $\tilde u$ are in general different. We obtain the map: $L\colon\Up_*\to\Up_*$, $L(u)=\tilde u$.

In the similar way we define ``turn to the right in the triangle $(\Pp,\omega\Pp,u\omega\Pp)$''. Since by transversality we have $\omega u\omega=\hat u^{-1}\omega\hat q$ for $\hat u\in\Up_*$ and $\hat q\in \Pp$. We define $T(v''\gets v)\coloneqq\hat u\omega$. Then
$$T(v'')F^t(v'')=T(v'')F_3=T(v''\gets v)\omega\Pp=\hat u\omega \omega \Pp=\Pp,$$
$$T(v'')F^b(v'')=T(v'')F_2=T(v''\gets v)u\omega \Pp=\hat u\omega u\omega \Pp=\hat u\hat u^{-1}\omega \hat q\Pp=\omega \Pp.$$
$$T(v'')F_1 =T(v''\gets v)\Pp=\hat u\omega \Pp,$$
We call the transformation $T(v''\gets v)=:T_r(u)$ ``turn to the right in the triangle $(\Pp,\omega\Pp,u\omega\Pp)$'' (see Figure~\ref{Turn_figure}). Similarly to the map $L$, we define the map $R\colon\Up_*\to\Up_*$, $R(u)=\hat u$.

So we have constructed a transverse framed $G$-local system on $\Gamma$.

\begin{rem}
An easy calculation shows that $R(L(u))=L(R(u))=u$. The set of fixed points of $L$ in $\Up_*$ is
$$\Fix_{\Up_*}(L)=\{u\in\Up_*\mid (u\omega)^3\in\Pp\}=\{u\in\Up_*\mid (u\omega)^3\in\Norm_\L(u)\}.$$
The last equality holds by Proposition~\ref{uom_k}. Notice, that in general $L^3(u)\neq u$. The equality $L^3(u)= u$ holds if and only if $u\in \Fix_{\Up_*}(L)$.
\end{rem}

\subsubsection{Quadrilateral. Crossing an edge of triangulation}\label{Quadruple}

Let $S=\tilde S$ be an ideal quadrilateral with cyclically ordered vertices $p_1,p_4,p_2,p_3$. Let $\T$ be the triangulation that consists of all external edges of $S$ and one internal edge connecting $p_1$ and $p_2$.

\begin{figure}[ht]
{\begin{tikzpicture}

\draw (-3,3) node[label = $p_1$] (v1) {} -- (-3,0) node[label = below:$p_2$] (v2) {} -- (0,1.5) node[label = right:$p_3$] (v3) {} -- (-3,3);
\draw (-3,0) -- (-6,1.5) node[label = left:$p_4$] (v4) {} -- (-3,3);





\end{tikzpicture}}
\caption{Quadrilateral $(p_1,p_4,p_2,p_3)$}\label{Quadr_figure0}
\end{figure}
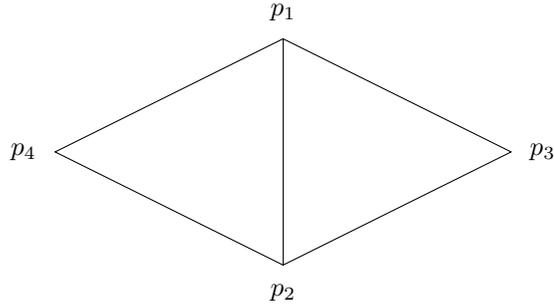

Let $(\rho,F)$ be a $\T$-transverse framed homomorphism. Since $\pi_1(S)$ is trivial, $\rho$ is trivial. Let $F(p_i)=F_i$ for all $i\in\{1,\dots,4\}$. Let the graph $\Gamma=\tilde\Gamma$ consist of two vertices $v$ and $v'$ and one edge $(v'\gets v)$ or $(v\gets v')$. Without loss of generality, we consider the first one.

\begin{figure}[ht]
\scalebox{1}{
\begin{tikzpicture}

\draw (-3,3) node[label = $\mathcal P^+$] (v1) {} -- (-3,-2) node[label = below:$\omega\mathcal P^+$] (v2) {} -- (0,0.5) node[label = left:$u\omega\mathcal P^+$] (v3) {} -- (-3,3);
\draw (-3,-2) -- (-6,0.5) node[label = right:$\omega u'\omega \mathcal P^+$] (v4) {} -- (-3,3);

\draw (-3+7,3) node[label = $\mathcal P^+$] (v1) {} -- (-3+7,-2) node[label = below:$\omega\mathcal P^+$] (v2) {} -- (0+7,0.5) node[label = left:$u'\omega\mathcal P^+$] (v3) {} -- (-3+7,3);
\draw (-3+7,-2) -- (-6+7,0.5) node[label = right:$\omega u\omega \mathcal P^+$] (v4) {} -- (-3+7,3);

\node[fill, circle, inner sep = 1.5pt, label= $v$](a) at (-2.5,0.5) {};
\node[fill, circle, inner sep = 1.5pt, label= $v'$](b) at (-3.5,0.5) {};
\draw[middlearrow={latex}] (a) to (b);

\node[fill, circle, inner sep = 1.5pt, label= $v'$](b) at (-2.5+7,0.5) {};
\node[fill, circle, inner sep = 1.5pt, label= $v$](a) at (-3.5+7,0.5) {};
\draw[middlearrow={latex}] (a) to (b);

\draw[middlearrow={latex}] (-.25,1.5) to[bend left] (1.5,1.5) ;
\draw[middlearrow={latex}] (1.5,-0.5) to[bend left] (-.25,-0.5) ;

\node[label = $\omega$] at (0.65,1.85) {};
\node[label = $\omega^{-1}$] at (0.65,-1.5) {};

\end{tikzpicture}}
\caption{Crossing an edge of triangulation}\label{Quadr_figure}
\end{figure}
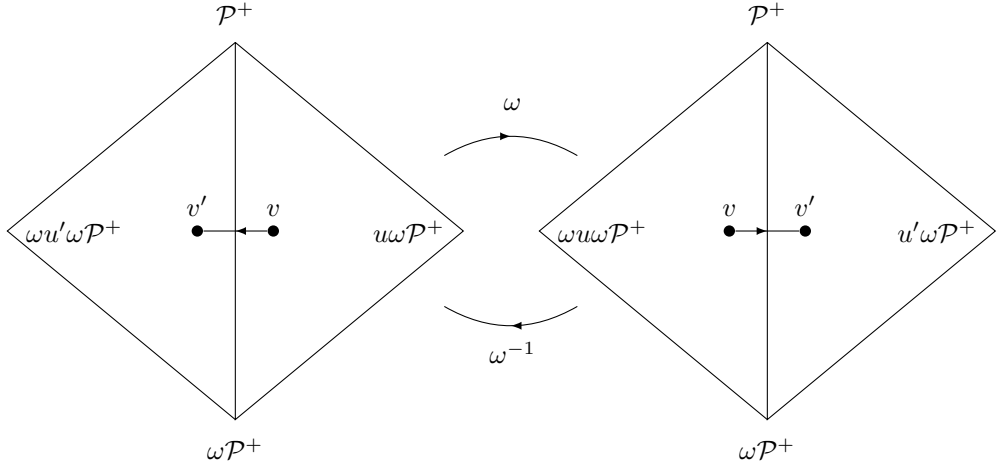

Because of transversality by Proposition~\ref{prop:quadr}, there exists an element $g\in G$ such that
$$g(F_1,F_4,F_2,F_3)=(\Pp,\omega u' \omega \Pp,\omega\Pp,u\omega\Pp)$$
with $u,u'\in\Up_*$ (see Figure~\ref{Quadr_figure}). We define $T(v)\coloneqq g$, $T(v'\gets v)\coloneqq\omega$, then $T(v')=\omega g$. The $G$-local system $(\Gamma,T)$ has a framing $F^t(v)=F^b(v')=F_1$ and $F^t(v')=F^b(v)=F_2$. So we have constructed a transverse framed $G$-local system on $\Gamma$.

The element $g$ is unique up to the left multiplication by an element of $\L$. Therefore:

\begin{cor}
The space $\Rep^{\fr}_\mathcal T(\pi_1(S),G)$ is homeomorphic to $(\Up_*)^2/\L$.
\end{cor}

Let now $S$ be any punctured surface, $\mathcal T$ be an ideal triangulation of $S$ and $(\rho,F)\in\Hom^{\fr}_{\mathcal T}(\pi_1(S),G)$. Let $\tilde\Gamma$ be the graph as in Section~\ref{sec:Graph_Gamma} with the framing $(F^t,F^b)$ induced by $F$. We can assume that the base point of the fundamental group agrees with some of point $b\in V_\Gamma$ that lies in a triangle $\tau$. We take $\tilde b\in V_{\tilde\Gamma}$ a lift of $b$.

Starting at $\tilde b$ and applying inductively the procedure described in Section~\ref{Turn} for every triangle and in Section~\ref{Quadruple} for every pair of adjacent triangles, we obtain a transverse framed local system on $\tilde\Gamma$. This local system has the property that the representation described in~(\ref{Graph_representation}) agrees with $\rho$. However, this local system is in general not $\pi_1(S)$-invariant, although the framing is $\rho$-equivariant. In the proof of Theorem~\ref{main_thm} in Section~\ref{section:parametrization}, we will modify this local system to make it $\pi_1(S)$-invariant.

\section{Parametrization of the space of framed representations}\label{section:parametrization}

We are now ready to prove one of the main results of this paper. As before, we assume $G$ to be a semisimple Lie group, $\Pp$ to be a self-opposite parabolic subgroup and $\L$ be the Levi factor corresponding to $\Pp$ and $\F=G/\Pp$ to be the flag variety associated with $\Pp$. Further, $S$ is an orientable surface with Euler characteristic $\chi(S)$ and with $p_e$ external punctures. The next theorem gives a parametrization of the space of transverse framed representations.

\begin{teo}\label{main_thm}
The space $\Rep_\mathcal T^{\fr}(\pi_1(S,b),G)$ is homeomorphic to the following space:
$$((\Up_*)^{p_e-2\chi(S)}\times\L^{1-\chi(S)})/\L.$$
\end{teo}

\begin{proof} 

Let $\rho\colon\pi_1(S,b)\to G$ be a $\T$-transverse homomorphism with a framing $F\colon \tilde P\to \F$ where $\T$ is an ideal triangulation of $S$. The framing $F$ induces a framing $(F^t,F^b)$ of $\tilde\Gamma$.

Let $b\in S$ be the base point of $\pi_1(S,b)$ and $\tilde b\in\tilde S$ be a lift of $b$. As before, we assume $\tilde b\in V_{\tilde\Gamma}$. Let $g\in G$ be such that $g(F^t(\tilde b),F^b(\tilde b))=(\Pp,\omega\Pp)$. We define
$$X\coloneqq\{(\rho,F,g)\mid (\rho,F)\in \Hom^{\fr}_{\mathcal T}(\pi_1(S,b),G),\; g\in G: g(F^t(\tilde b),F^b(\tilde b))=(\Pp,\omega\Pp)\}$$
and the equivalence relation $\sim$ on $X$ as follows: $(\rho,F,g)\sim (\rho',F',g')$ if and only if $(g\rho g^{-1},gF)=(g'\rho' g'^{-1},g'F')$.

Step 1: Assume $S=\bar S$ is a triangulated polygon and $\Gamma$ defined as in Section~\ref{sec:Graph_Gamma}. Let $(\rho,F,g)\in X$. We define $T(b)=g$ and then apply the procedure defined in Section~\ref{Moves}. We obtain in this way a framed transverse $G$-local system on $\Gamma$. Let $v\in V_\Gamma^0$. $v$ lies in the triangle with vertices $v^t,v^b,v^r\in P$. Let $u_v\in\Up_*$ be defined as follows $T(v)F(v^r)=u_v\omega\Pp$. The elements $u_v$ are uniquely defined for all $v\in V_\Gamma^0$.

Therefore, we obtain the following map:
$$\begin{matrix}
\mathcal X\colon & X & \to & (\Up_*)^{V_\Gamma^0}\\
 & (\rho,F,g) & \mapsto & (u_v)_{v\in V_\Gamma^0}.
\end{matrix}$$
Notice that $\mathcal X(\rho,F,g)=\mathcal X(\rho',F',g')$ if and only if  $(\rho,F,g)\sim(\rho',F',g')$. This map is surjective because any choice of $(u_v)_{v\in V_\Gamma^0}\in (\Up_*)^{V_\Gamma^0}$ defines a transverse framed $G$-local system on $\Gamma$ that by~\ref{Loc_to_Rep} defines a framed homomorphism $(\rho,F)\in\Hom^{\fr}_\T(\pi_1(S),G)$.

The $\L$-action on $X$ by the left multiplication in the last component and by componentwise conjugation on $(\Up_*)^{V_\Gamma^0}$ that is compatible with $\mathcal X$, i.e.\ $\mathcal X(\rho,D,lg)=(lu_vl^{-1}\mid v\in V_\Gamma^0)$.
We obtain the quotient map:
$$\bar{\mathcal X}\colon \Rep^{\fr}_{\mathcal T}(\pi_1(S,b),G)\to (\Up_*)^{V_\Gamma^0}/\L$$
which is a homeomorphism.

Step 2: Now let us turn to the general case. We choose a connected fundamental domain $S_0$ for $S$ in $\tilde S$ that completely consists of ideal triangles and $\tilde b$ lies in one of them. The fundamental domain $S_0$ is an ideal polygon, so we can apply the procedure above and obtain a transverse framed $G$-local system over $\Gamma_0$ which is a subgraph of $\tilde \Gamma$ contained in $S_0$. We want to extend this local system to the entire graph $\tilde\Gamma$ in a $\pi_1(S)$-invariant way.

To obtain a~$\pi_1(S)$-invariant local system, according to the Section~\ref{Loc_to_Rep}, we need to define $T(\tilde v)$ for every $\tilde v\in V_\Gamma$ as follows: let $\gamma\in\pi_1(S)$ be the unique element such that $\gamma v\in S_0$. We define $T(\tilde v)\coloneqq T(\gamma \tilde v)\rho(\gamma)$. From this follows immediately that for every $\tilde v\in V_{\tilde\Gamma}$ and for every $\gamma\in\pi_1(S)$, $T(\gamma\tilde v)=T(\tilde v)\rho(\gamma)^{-1}$.

The local system $(\tilde\Gamma,T)$ is indeed $\pi_1(S)$-invariant because: let $(w\gets v)\in E_{\tilde\Gamma}$ and $\gamma\in\pi_1(S)$. Then $$T(\gamma\tilde w\gets\gamma\tilde v)=T(\gamma\tilde w)T(\gamma\tilde v)^{-1}=T(\tilde w)\rho(\gamma)^{-1}\rho(\gamma)T(\tilde v)^{-1}=T(\tilde w)T(\tilde v)^{-1}=T(\tilde w\gets\tilde v).$$

The framing $(F^t,F^b)$ of $(\tilde\Gamma,T)$ is adapted to $(\tilde\Gamma, T)$. Indeed, if $\tilde v$ is contained in $S_0$, then by construction $T(\tilde v)(F^t(\tilde v),F^b(\tilde v))=(\Pp,\omega\Pp)$. If $\tilde v\in V_{\tilde\Gamma}$ any vertex, then there exist the unique $\gamma\in\pi_1(S)$ with $\gamma\tilde v$ is contained in $S_0$. Then
$$T(\tilde v)(F^t(\tilde v),F^b(\tilde v))=T(\gamma \tilde v)\rho(\gamma)\rho(\gamma)^{-1}(F^t(\gamma\tilde v),F^b(\gamma\tilde v))=T(\gamma \tilde v)(F^t(\gamma\tilde v),F^b(\gamma\tilde v))=(\Pp,\omega\Pp).$$

The framed local system $(\tilde \Gamma,T, F^t,F^b)$ is determined by its restriction to $\Gamma_0$ and the maps $T(\tilde w\gets \tilde v)$ for $\tilde v\in S_0$ and $\tilde w\notin S_0$. In general $T(\tilde w\gets \tilde v)\neq\omega$ although the edge $(\tilde w\gets \tilde v)$ intersects an edge of triangulation $\tilde\T$. We now want to understand the maps $T(\tilde w\gets \tilde v)$.

Let $e,e'$ be edges of the triangulation $\tilde\T$ that are in the boundary of $S_0$ but not in the boundary of $\tilde S$ and let $\gamma\in\pi_1(S)$ the unique element such that $\gamma(e)=e'$. Without loss of generality assume $(\tilde w\gets \tilde v)\in E^+_{\tilde\Gamma}$ is an edge that crosses $e$ and $\tilde v\in S_0$, $\tilde w\notin S_0$. Then the edge $(\gamma_i\tilde w\gets \gamma_i\tilde v)\in E^+_{\tilde\Gamma}$ crosses $e'$ and $\gamma_i\tilde v\notin S_0$, $\gamma_i\tilde w\in S_0$.
$$T(\tilde w\gets \tilde v)= T(\tilde w)T(\tilde v)^{-1}=T(\gamma \tilde w)\rho(\gamma)T(\tilde v)^{-1}.$$
Therefore:
$$T(\tilde w\gets \tilde v)(\Pp,\omega\Pp)
=T(\gamma \tilde w)\rho(\gamma)T(\tilde v)^{-1}(\Pp,\omega\Pp)
=T(\gamma \tilde w)\rho(\gamma)(F^t(\tilde v),F^b(\tilde v))$$
$$=T(\gamma \tilde w)(F^t(\gamma\tilde v),F^b(\gamma\tilde v))
=T(\gamma \tilde w)(F^b(\gamma\tilde w),F^t(\gamma\tilde w))
=(\omega\Pp,\Pp).$$
This means that there exist a unique $\ell_e\in\L=\Stab_G(\Pp,\omega\Pp)$ such that $T(\tilde w\gets \tilde v)=\ell_e\omega$. So for every pair $\{e,e'\}$ of external edges that are related by a nontrivial element of $\pi_1(S)$ we get an element of $\L$. We denote the set of such pairs $E_0$. Its cardinality is $1-\chi(S)$.

As in Step 1, we obtain a map
$$\begin{matrix}
\mathcal X\colon & X & \to & (\Up_*)^{V_\Gamma^0}\times\L^{E_0}\\
 & (\rho,F,g) & \mapsto & (u_v,l_e)_{v\in V_\Gamma^0,\;e\in E_0}.
\end{matrix}$$
Notice that $\mathcal X(\rho,F,g)=\mathcal X(\rho',F',g')$ if and only if  $(\rho,F,g)\sim(\rho',F',g')$. Further, this map is surjective because any choice of $(u_v,l_e)_{v\in V_\Gamma^0,\;e\in E_0}\in (\Up_*)^{V_\Gamma^0}\times\L^{E_0}$ defines a $\pi_1(S)$-invariant transverse framed $G$-local system on $\tilde\Gamma$ that by~\ref{Loc_to_Rep} defines a framed homomorphism $(\rho,F)\in\Hom^{\fr}_\T(\pi_1(S),G)$.

The $\L$-action on $X$ by the left multiplication in the last component and by componentwise conjugation on $(\Up_*)^{V_\Gamma^0}\times\L^{E_0}$ that is compatible with $\mathcal X$, i.e.\ $\mathcal X(\rho,D,lg)=(lu_vl^{-1},ll_el^{-1}\mid v\in V_\Gamma^0,\;e\in E_0)$.
We obtain the quotient map:
$$\bar{\mathcal X}\colon \Rep^{\fr}_{\mathcal T}(\pi_1(S,b),G)\to ((\Up_*)^{V_\Gamma^0}\times\L^{E_0})/\L$$
which is a homeomorphism.
\end{proof}

\section{Space of positive framed representations and its parametrization}\label{sec:positive}

So far, we only required the framed representation to be transverse with respect to the triangulation.In this section, we want to understand the space of positive framed representation in a similar way as we did it for the transverse ones in the previous section. For this, we assume $G$ to be a semisimple Lie group with a $\Theta$-positive structure. As before, we take $\Pp$ the parabolic subgroup corresponding to $\Theta$ and $\L_0$ the subgroup of the Levi subgroup $\L$ stabilizing the positive semigroup $\Upp$. We also consider the flag variety $\F\coloneqq G/\Pp$ as before. Further, $S$ is an orientable surface with Euler characteristic $\chi(S)$ and with $p_e$ external punctures (cf.~Section~\ref{sec:top_data}).

\begin{df} A framing $F$ is called \defin{positive} if for every cyclically oriented quadruple $(\tilde p_1,\dots,\tilde p_4)\in\tilde P^4$, the quadruple of flags $(F(\tilde p_1),\dots,F(\tilde p_4))$ is positive.
A homomorphism $\rho\in\Hom(\pi_1(S),G)$ is called \defin{positive} if it admits a positive framing $F$. The pair $(\rho, F)$ is called a \defin{positive framed homomorphism}.

The space of all positive framed homomorphisms is denoted by $\Hom^{\fr}_+(\pi_1(S),G)$. The space
$$\Rep^{\fr}_+(\pi_1(S),G)\coloneqq \Hom^{\fr}_+(\pi_1(S),G)/G$$
is called the \defin{space of positive framed representations}.
\end{df}

\begin{rem}
Every positive framed representation is transverse with respect to any ideal triangulation.
\end{rem}

\begin{rem}
The space of positive framed representations of an ideal $k$-gon is precisely $\Conf_k^+(\F)$.
\end{rem}

To simplify the notation, we say that a (positively oriented) quadrilateral $(\tilde p_1,\tilde p_2,\tilde p_3,\tilde p_4)\in\tilde P^4$ of $\tilde{\mathcal T}$ is \defin{positive with respect to the framing} $F$, if $(F(\tilde p_1),F(\tilde p_2),F(\tilde p_3),F(\tilde p_4))$ is a positive quadruple of flags.

\begin{prop}
Let $\mathcal T$ be a triangulation of $S$ and let $F$ be a framing such that for every positively oriented quadrilateral $(\tilde p_1,\tilde p_2,\tilde p_3,\tilde p_4)\in\tilde P^4$ of $\tilde{\mathcal T}$  the quadruple of flags $(F(\tilde p_1),F(\tilde p_2),F(\tilde p_3),F(\tilde p_4))$ is positive. Then the framing $F$ is positive.
\end{prop}

\begin{proof}
For every quadruple $(\tilde q_1,\tilde q_2,\tilde q_3,\tilde q_4)\in\tilde P^4$ there exists a (triangulated, finite) polygon $\Pi$ in $\tilde S$ with vertices in $\tilde P$ and edges in $\tilde{\mathcal T}$ such that $\tilde q_1,\tilde q_2,\tilde q_3,\tilde q_4$ are vertices of $\Pi$. There exist a triangulation $\tilde{\mathcal T'}$ of $\Pi$ containing $(\tilde q_1,\tilde q_2,\tilde q_3,\tilde q_4)$ as a quadrilateral. Any two ideal triangulation of $\Pi$ are connected by a sequence of flips. Therefore we need to check the following statement: if $\tilde{\mathcal T}$ and $\tilde{\mathcal T'}$ differ by a flipping of one edge, and all quadrilaterals of $\tilde{\mathcal T}$ are positive with respect to $F$ then all quadrilaterals of $\tilde{\mathcal T'}$ are positive with respect to $F$ as well.


Let $Q\coloneqq (\tilde p_1,\tilde p_2,\tilde p_3,\tilde p_4)$ be a quadrilateral in $\tilde{\mathcal T}$ and the edge $(\tilde p_1,\tilde p_3)$ is in $\tilde{\mathcal T}$. We do the flip in this quadrilateral, i.e.\ we exchange the edge $(\tilde p_1,\tilde p_3)$ by the edge $(\tilde p_2,\tilde p_4)$. Quadrilaterals that do not intersect $Q$ are not affected by this flip.  So, without loss of generality, we take a vertex  $\tilde p_5$  such that the triangle $(\tilde p_1,\tilde p_2,\tilde p_5)$ is in $\tilde{\mathcal T}$. We need to check if the quadruple $(F(\tilde p_1),F(\tilde p_5),F(\tilde p_2),F(\tilde p_4))$ is positive.

Since the quadrilaterals $(\tilde p_1,\tilde p_2,\tilde p_3,\tilde p_4)$ and $(\tilde p_1,\tilde p_5,\tilde p_2,\tilde p_3)$ are in $\tilde{\mathcal T}$, the quadruples $(F(\tilde p_1),F(\tilde p_2),F(\tilde p_3),F(\tilde p_4))$ and $(F(\tilde p_1),F(\tilde p_5),F(\tilde p_2),F(\tilde p_3))$ are positive, i.e.\;there exists an element $g\in G$ such that
$$g(F(\tilde p_2),F(\tilde p_3),F(\tilde p_4),F(\tilde p_1))=(\Pp,\omega\Pp,u_1\omega\Pp,u_1u_2\omega\Pp)$$
where $u_1,u_2\in\Upp$. Further, by Fact~\ref{properties_positive_tuple},
$$g(F(\tilde p_2),F(\tilde p_3),F(\tilde p_1),F(\tilde p_5))= (\Pp,\omega\Pp,u_1u_2\omega\Pp,u_1u_2u_3\omega\Pp)$$
for some $u_3\in \Upp$. Therefore the 5-tuple $(F(\tilde p_1),F(\tilde p_5),F(\tilde p_2),F(\tilde p_3),F(\tilde p_4))$ is positive. In particular, the quadruple $(F(\tilde p_1),F(\tilde p_5),F(\tilde p_2),F(\tilde p_4))$ is positive.
\end{proof}

\begin{cor}
To show that a framing $F$ is positive, it is enough to check that for some ideal triangulation all positively oriented quadrilaterals of this triangulation are positive with respect to $F$.
\end{cor}

The following theorem gives a parametrization of the space of positive framed representations.

\begin{teo}\label{main_thm_positive}
The space $\Rep_+^{\fr}(\pi_1(S,b),G)$ of positive framed representations is homeomorphic to the following space:
$$((\Upp)^{p_e-2\chi(S)}\times\L_0^{1-\chi(S)})/\L_0.$$
Let $n_\beta$ be the number of times that the space $\up_\beta$ appears in the domain of definition of the~map $F$ defined in Section~\ref{positive_groups} and $K_0$ be a maximal compact subgroup of $\L_0$.
Then the space $\Rep_+^{\fr}(\pi_1(S,b),G)$ is homeomorphic to the following spaces:
$$\left(\prod_{\beta\in\Theta}(\mathring c_\beta)^{n_\beta(p_e-2\chi(S))}\times\L_0^{1-\chi(S)}\right)/\L_0\,,\quad\left(\prod_{\beta\in\Theta}(\mathring c_\beta)^{n_\beta(p_e-2\chi(S)) -\chi(S)}\times K_0^{1-\chi(S)}\right)/K_0$$
where $\L_0$ and $K_0$ act by the adjoint action on every $\mathring c_\beta$ and by conjugation on $\L_0$.
\end{teo}

The proof of the first statement of the theorem essentially agrees with the proof of Theorem~\ref{main_thm}. Using Proposition~\ref{Levi-symmetric_space} and the polar decomposition for $\L_0$, we obtain the last statement.

\section{Homotopy type of the space of positive framed representations}\label{sec:homotopy_type}

The goal of this section is to describe the homotopy type of the space of positive framed representation. More precisely, we show that the space of positive framed representations admits a strong deformation retract that is a quotient of a compact manifold without boundary by an action of a compact group.

To do this a key point is to introduce the space of \defin{degenerate representations}. For this, we consider the points $(u_\Theta,\dots,u_\Theta,k_1,\dots k_{1-\chi(S)})\in((\Upp)^{p_e-2\chi(S)}\times\L_0^{1-\chi(S)})$ where $u_\Theta$ is the special element in $\Upp$ defined in~(\ref{special_u}) and all $k_i\in\Norm_{G}(u_\Theta)$. We remind that $u_\Theta=F(v_{\beta_{i_1}},\dots,v_{\beta_{i_l}})$ for fixed elements $v_\beta\in \mathring c_\beta$ for all $\beta\in \Theta$. Notice that $K_0\coloneqq \Norm_{G}(u_\Theta)\subseteq \L_0$ is a maximal compact subgroup of $\L_0$. A representation with parameters as above is called \defin{degenerate}. We denote by $\mathcal D$ the subspace of all degenerate representations in $\Rep_+^{\fr}(\pi_1(S,b),G)$. A direct computation using Proposition~\ref{properties_special_u} shows that holonomies around punctures for a degenerate representation with parameters as above are conjugated to $ku_\Theta^s$ for some $s\in\Z\bs\{0\}$ and $k\in \Norm_{\L_0}(u_\Theta)$.

\begin{prop}
Let $u\in\Up$ and $u_\Theta\in\Upp$ as above. There exists $N=N(u)\in\N$ such that $u_\Theta^nu,uu_\Theta^n\in\Upp$ for all $n\geq N$.
\end{prop}

\begin{proof}
We prove that $uu_\Theta^N\in\Upp$. The second inclusion can be proven similar. Without loss of generality, we assume $G$ to be the adjoint form. Let $q\colon\SL_2(\R)\to G$ be the positive principal homomorphism which maps upper triangular matrices to $\Pp$, lower triangular matrices to $\Pm$ and such that $q\begin{pmatrix}1 & 1 \\ 0 & 1\end{pmatrix}=u_\Theta$ and $q\begin{pmatrix}0 & 1 \\ -1 & 0\end{pmatrix}=\omega_\Theta$. Let $\ell_k\coloneqq q\begin{pmatrix}\sqrt{k} & 0 \\ 0 & \sqrt{k^{-1}}\end{pmatrix}\in\L_0$ where $k\in\N$. Then $\ell_ku_\Theta \ell_k^{-1}=q\begin{pmatrix}1 & k \\ 0 & 1\end{pmatrix}=u_\Theta^k$.

Further, $\ell_k^{-1} u u^k_\Theta \ell_k=\ell_k^{-1} u \ell_k u_\Theta$. Since $q$ is a principal positive homomorphism, $\ell_k^{-1} u \ell_k$ converges to the identity element in $\Up$. Since $\Upp$ is open in $\Up$, there exists $N\in\N$ such that $\ell_n^{-1} u u^n_\Theta \ell_n=\ell_n^{-1} u \ell_n u_\Theta\in\Upp$ for all $n\geq N$. Since $\ell_N\in\L_0=\Norm_\L(\Upp)$, $u u^N_\Theta\in\Upp$.
\end{proof}

\begin{prop}
The element $ku^s_\Theta$ for some $s\in\Z\bs\{0\}$ and $k\in \Norm_{G}(u_\Theta)$ has $\Pp$ as the unique invariant flag in $\F$. In particular, any degenerate representation has exactly one framing.
\end{prop}

\begin{proof}
Let $ux\Pp$ be an invariant flag of $ku^s_\Theta$ where $x$ is a lift of a element of the group $W(\Theta)$ and $u\in\Up$.  Let $n\in\N$ be such that $sn\geq N(u)$ where $N(u)$ is the number from the previous proposition. Then 
\begin{align*}
ux\Pp & =(k u^{s}_\Theta)^n ux\Pp=k^n u^{sn}_\Theta ux\Pp=k^n u^{sn}_\Theta u k^{-n}k^n x\Pp\\
&=k^n u^{sn}_\Theta u k^{-n}x x^{-1}k^n x\Pp=k^n u^{sn}_\Theta u k^{-n}x\Pp.
\end{align*}
By assumption on $n$, $u^{sn}_\Theta u\in \Upp$. Since $K_0$ stabilizes $\Upp$, $k^n u^{sn}_\Theta u k^{-n}\in\Upp$. So for the further consideration, we can assume $u\in\Upp$.

Now let $m\in\N$ be such that $sm\geq N(u^{-1})$. Then the equality $(ku^s)^m ux\Pp=ux\Pp$ is equivalent to $x^{-1}u^{-1}(ku^s)^m ux\in\Pp$. Further,
$$x^{-1}u^{-1}(ku^s)^m ux=x^{-1}u^{-1}u^{sm}k^muk^{-m}k^mx.$$
The elements $u^{-1}u^{sm}$ and $k^muk^{-m}$ are in $\Upp$, and $k^mx=x x^{-1}k^mx\in x\Pp$. We denote $\tilde u\coloneqq u^{-1}u^{sm}k^muk^{-m}\in\Upp$. Then $x^{-1}\tilde u x\in\Pp$. But if the element $x$ is not a lift of the trivial element of $W(\Theta)$, then it sends some positive root spaces to negative root spaces, and all projections of elements of $\Upp$ to positive root spaces are non-trivial. Therefore, $x^{-1}\tilde u x\in \Pp$ if~and~only~if $x$ is a lift of the trivial element of $W(\Theta)$, i.e.\ $x\Pp=\Pp$.
\end{proof}

\begin{prop}\label{retraction}
The subspace $\mathcal D\subset\Rep_+^{\fr}(\pi_1(S,b),G)$ is a strong deformation retract of $\Rep_+^{\fr}(\pi_1(S,b),G)$. In particular, the strong deformation retraction provides a bijection on the level of connected components of $\Rep_+^{\fr}(\pi_1(S,b),G)$ and $\mathcal D$, and $\Rep_+^{\fr}(\pi_1(S,b),G)$ is homotopy equivalent to $K_0^{1-\chi(S)}/K_0$.
\end{prop}

\begin{proof}
To prove the statement, we need to construct a strong deformation retraction $H\colon \Rep_+^{\fr}(\pi_1(S,b),G)\times [0,1]\to \Rep_+^{\fr}(\pi_1(S,b),G)$ such that $H(x,t)=x$ for all $x\in\mathcal D$ and all $t\in [0,1]$, $H(x,1)=x$ for all $x\in \Rep_+^{\fr}(\pi_1(S,b),G)$ and $H(x,0)\in \mathcal D$ for all $x\in \Rep_+^{\fr}(\pi_1(S,b),G)$.


Since the action of $K_0$ preserves the decomposition into factors of
$$\prod_{\beta\in\Theta}(\mathring c_\beta)^{(p_e-2\chi(S)) n_i-\chi(S)}\times K_0^{1-\chi(S)},$$
we define the $K_0$-equivariant retraction on every factor. Since every $\mathring c_\beta$ is a convex cone, $(v,t)\mapsto vt+v_\beta(1-t)\in \mathring c_\beta$ for every $v\in \mathring c_\beta$ and for all $t\in [0,1]$ is a well-defined retraction. Since conjugation by elements of $K_0$ is linear fixing $v_\beta$, this retraction is $K_0$-equivariant. 

Combining these retractions, we get a retraction of $\prod_{\beta\in\Theta}(\mathring c_\beta)^{(p_e-2\chi(S)) n_i-\chi(S)}\times K_0^{1-\chi(S)}/K_0$ onto $\mathcal D=\{u_\Theta\}^{p_e-2\chi(S)}\times K_0^{1-\chi(S)}/K_0=\left(\prod_{\beta\in\Theta}\{v_\beta\}^{n_i}\right)^{p_e-2\chi(S)}\times K_0^{1-\chi(S)}/K_0$ that is homeomorphic to $K_0^{1-\chi(S)}/K_0$.
\end{proof}

\section{Connected components of the space of positive representations}\label{sec:connected_components}

In this section, we study the space of positive non-framed representations. Using the results for framed representations, we show that for connected Lie group with positive structure the spaces of framed and non-framed positive representations has the same number of connected components.

Let $\Rep^{\fr}(\pi_1(S),G)\to \Rep(\pi_1(S),G)$ be the natural projection. In this section we call this projection the \defin{framing forgetting map}. We denote by $\Rep_+(\pi_1(S),G)$ the image of $\Rep_+^{\fr}(\pi_1(S),G)$ under the framing forgetting map.

First, we recall the notion of the path lifting property: A continuous map $f\colon X\to Y$ is said to have the \defin{path lifting property} if for every continuous $\sigma\colon [0, 1] \to Y$ and every $x\in f^{-1}(\sigma(0))$ there exists, up to
reparametrization, a lift $\tilde \sigma$ of $\sigma$ starting at $x$, namely $\tilde\sigma\colon [0, 1] \to X$ is continuous, $\tilde\sigma(0) = x$ and
there is a continuous, increasing, surjective function $\psi\colon [0, 1] \to [0, 1]$ such that $f\circ\tilde\sigma= \sigma\circ\psi$.
This is the case, for example, when $f$ is a covering; this is also the case when, for any $\sigma$ as above, the space $\sigma^*(f) = \{(t, x) \in [0, 1] \times X \mid \sigma(t) = f(x)\}$ is path-connected. A piecewise linear map between simplicial spaces that is surjective and with connected ﬁbers has the path lifting property.

\begin{prop}
The framing forgetting map $\Rep_+^{\fr}(\pi_1(S),G)\to\Rep_+(\pi_1(S),G)$ has the path lifting property.
\end{prop}
\begin{proof}
Let $\Hom_+(\pi_1(S), G))$ and $\Hom^{\fr}_+(\pi_1(S), G))$ the spaces of positive homomorphisms and of framed positive homomorphisms. To prove the proposition we essentially follow the strategy from~\cite[Section~7]{AGRW}. The
proposition is a direct consequence of the following lemmas.

\begin{lem}\label{lem:path-lift-prop-hom-to-max}
The map $p\colon \Hom_+(\pi_1(S), G)\to \Rep_+(\pi_1(S), G)$ has the path lifting property.
\end{lem}

\begin{proof} The fibers of $\sigma^*(p)$ are path connected, since they are orbits for the action of the (path) connected group $G$, and thus the total space $\sigma^*(p)$ is indeed connected.
\end{proof}

\begin{lem}  \label{lem:path-lift-prop-dec-to-max-hom}
The map $\Hom^{\fr}_+(\pi_1(S), G)) \to  \Hom_+(\pi_1(S), G))$ has the path lifting property.
\end{lem}

\begin{proof}
Since we can choose independently the framing at the punctures, it is enough to answer positively the following problem:
  \begin{quote}
    Let $\mathfrak{P}\subset G$ be the set of elements $g$ having at
    least one invariant flag and $\mathfrak{Q}=\{(g,L) \in
      \mathfrak{P}\times \F \mid g\cdot L=L\}$. Then the natural map
      $\pi\colon \mathfrak{Q} \to \mathfrak{P}$ has the path lifting property.
  \end{quote}
Let $\Pp$ be the parabolic subgroup of $G$ as before considered also as an element of $\F$. Its stabilizer in $G$ agrees with $\Pp$ itself. We then have a surjective map  $G \times \Pp \to \mathfrak{Q}$ defined as $(h,p) \mapsto (hph^{-1}, h\Pp)$ and it is enough to prove that its composition with~$\pi$ has the path  lifting property, i.e.\ that $G \times \Pp \to \mathfrak{P}$ with $(h,p) \mapsto hph^{-1}$ has the path lifting property.

Since the center of $G$ acts trivially on $G$ by conjugation, this map factors through the map $G \times \Pp \to \Ad(G)\times \Pp\to\mathfrak{P}$, where the group $\Ad(G)\subseteq GL(\g)$ is the adjoint form of $G$. Since $G$ is connected, the natural map $G\to\Ad(G)$ is a connected covering and, in particular, has the path lifting property.

The map $\Ad(G)\times \Pp\to\mathfrak{P}$ is algebraic so that there exists simplicial structures on $\Ad(G)\times \Pp$ and on~$\mathfrak{P}$ such that the map is piecewise linear; since it is as well surjective, it has the path lifting property.
\end{proof}
This finishes the proof of the Proposition.
\end{proof}

\begin{teo}
The framing forgetting map $\Rep_+^{\fr}(\pi_1(S),G)\to\Rep_+(\pi_1(S),G)$ induces a bijection between $\pi_0(\Rep_+^{\fr}(\pi_1(S),G))$ and $\pi_0(\Rep_+(\pi_1(S),G))$.
\end{teo}

\begin{proof}
The map between $\pi_0(\Rep_+^{\fr}(\pi_1(S),G))$ and $\pi_0(\Rep_+(\pi_1(S),G))$ is clearly surjective. We need to show that it is injective.

Let $\mathcal D_0$ be the image of the space of degenerate representations $\mathcal D$ under the framing forgetting map. Restricted to $\mathcal D$, the ``framing forgetting'' map is a bijection, because all the representations of $\mathcal D$ have only one framing. In this way $\mathcal D$ embeds naturally into $\Rep_+(\pi_1(S),G)$. We need to show that this embedding is an injection on the level of connected components. Let $x,y\in\mathcal D_0$ be two elements in one path connected component of $\Rep_+(\pi_1(S),G)$. We take a path $\gamma\colon [0,1]\to \Rep_+(\pi_1(S),G)$ such that $\gamma(0)=x$, $\gamma(1)=y$. Because of the path lifting property of the framing forgetting map, we can lift if to a path $\gamma'\colon [0,1]\to \Rep^{\fr}_+(\pi_1(S),G)$ such that $\gamma'(0)=x$ and $\gamma(1)=y$ (here we identify $\mathcal D$ with its image $\mathcal D_0$ under the ``framing forgetting'' map). Using the strong deformation retraction constructed in Proposition~\ref{retraction}, we retract the path $\gamma'$ to a path $\gamma''$ that lies entirely in $\mathcal D$ and connects $x$ and $y$, i.e.\ $x$ and $y$ lie in the same connected component of $\mathcal D$.
\end{proof}

\section{Examples}\label{sec:examples}

\subsection{Trivial example} Let $G$ be any Lie group. We take as $\Pp=\L=G$  and $\Up=\{\Id\}$. The flag variety $\F$ is just one point $G/G$. In this case any representation is transverse to any ideal triangulation $\T$. The main theorem gives us then the following classical result:
$$\Rep^{\fr}_\T(\pi_1(S),G)=\Rep(\pi_1(S),G)=G^{1-\chi(S)}/G.$$

\subsection{\texorpdfstring{$G$}{G} is isogenic to \texorpdfstring{$\SL_2(\R)$}{SL(2,R)}} If $G=\SL_2(\R)$, $\Upp$ is homeomorphic to $\R_+$ and
$$\L_0=\L=\{\diag(a,a^{-1})\mid a\in\R^\times\}$$
is an abelian group isomorphic to $\R^\times$. Then $\Rep_+^{\fr}(\pi_1(S,b),G)$ is homeomorphic to $\R_+^{p_e-2\chi(S)}\times (\R^\times)^{1-\chi(S)}/\R^\times$ where $a\in\R^\times$ acts trivially on $\R^\times$-factors and acts by multiplication by $a^2\in\R_+$ on $\R_+$-factors. We can simplify this quotient space identifying the group $\R^\times$ with the multiplicative group $\R_+\times\{1,-1\}$. We obtain in this way:
$$\bigl(\R_+^{p_e-2\chi(S)}\times (\R^\times)^{1-\chi(S)}\bigr)/\R^\times=\R_+^{p_e-3\chi(S)}\times\{1,-1\}^{1-\chi(S)}.$$

If $G=\PSL(2,\R)$, then $\L_0=\L=\{\diag(a,a^{-1})\mid a>0\}$. Then in the same way we obtain that $\Rep_+^{\fr}(\pi_1(S,b),G)$ is homeomorphic
$\R_+^{p_e-3\chi(S)}$. So we obtain a parametrization of the classical (framed) Teichm\"uller space for $S$.

\subsection{\texorpdfstring{$G=\PSL_n(\R)$}{G=PSL(n,R)}} In this case $\Upp$ is homeomorphic to $\R_+^{\frac{n(n-1)}{2}}$ and
$$\L_0=\{\diag(a_1,\dots,a_n)\mid a_i\in\R_+,\;a_1\dots a_n=1\}$$ is an abelian group isomorphic to $\R_+^{n-1}$. Its maximal compact subgroup is trivial. 
So we obtain that the space of positive framed representations is homeomorphic to
$$\R_+^{(p_e-2\chi(S)){\frac{n(n-1)}{2}}-\chi(S)(n-1)}$$
If $S$ does not have boundary, and the number of punctures is equal to $p$, we obtain $(p_e-2\chi(S)){\frac{n(n-1)}{2}}-\chi(S)(n-1)=-\chi(S)(n^2-1)$. So we obtain a parametrization of the space of positive framed representations.

More generally, if $G$ is a split real adjoint Lie group, then the Levi factor $\L_0$ is always the connected component of the identity of the Cartan subgroup, i.e.\ $\L_0\cong\R_+^N$ for some appropriate $N>0$. Its maximal compact subgroup is trivial, so the spaces of positive representations for adjoint spit real groups are always homeomorphic to $\R_+^M$ for $M=-\chi(S)\dim(G)+p_e\dim(\Up)$ and, in particular, they are connected. So in the case $p_e=0$, we obtain the result of~\cite[Theorem~1.13(ii)]{FG}.

\subsection{\texorpdfstring{$G=\Sp_{2n}(\R)$}{G=Sp(2n,R)} with Hermitian positive structure}\label{ex:sp_2n} In this case $\Upp$ is homeomorphic to $\Sym^+_n(\R)$ and
$\L_0=\L$ is isomorphic to $\GL_n(\R)$, its maximal compact subgroup is $\OO(n)$. It acts on itself by conjugation and on $\Sym^+_n(\R)$ by similarities, i.e.\ for $g\in\GL_n(\R)$ and $B\in \Sym^+_n(\R)$, $g(B)=gBg^t$. 
Using this, we obtain that the space of maximal framed representations is homeomorphic to:
$$\bigl(\Sym^+_n(\R)^{p_e-2\chi(S)}\times \GL_n(\R)^{1-\chi(S)}\bigr)/\GL_n(\R)=\bigl(\Sym^+_n(\R)^{p_e-3\chi(S)}\times \OO_n(\R)^{1-\chi(S)}\bigr)/\OO_n(\R)$$
So we obtain the result of~\cite{AGRW}.

\subsection{\texorpdfstring{$G$}{G} is \texorpdfstring{$\GL_2(A)$}{GL(2,A)} and some of its subgroups} Let $A$ be a unital $\R$-algebra. We denote by $A^\times$ the group of all invertible elements of $A$. Let $G=\GL_2(A)$ and let $\Pp$ be the subgroup of $G$ of all upper-triangular matrices. Then $\L=\{\diag(a,a')\mid a,a'\in A^\times\}\cong (A^\times)^2$, $\Up=\left\{\begin{pmatrix}1 & b \\ 0 & 1 \end{pmatrix}\mid v\in A\right\}\cong A$ and $\Up_*=\left\{\begin{pmatrix}1 & b \\ 0 & 1 \end{pmatrix}\mid v\in A^\times\right\}\cong A^\times$. So we obtain that $\Rep^{\fr}(\pi_1(S),G)$ is homeomorphic to
$$\bigl((A^\times)^{p_e-2\chi(S)}\times (A^\times\times A^\times)^{1-\chi(S)}\bigr)/(A^\times\times A^\times)$$
where $A^\times\times A^\times$ acts on itself by conjugation and $A^\times\times A^\times$ acts on $A^\times$ as follows: for $(c,c')\in A^\times\times A^\times$, $b\in A^\times$, $(c,c').b = cb(c')^{-1}$.

This quotient space can be simplified as follows: in the quotient space two tuples $t\coloneqq ((b_i)_{i=1}^{p_e-2\chi(S)},(a_j,a_j')_{j=1}^{1-\chi(S)})$ and $t'\coloneqq ((cb_i(c')^{-1})_{i=1}^{p_e-2\chi(S)},(ca_jc^{-1},c'a_j'(c')^{-1})_{j=1}^{1-\chi(S)})$ where $a_j,a_j',b_i,c,c'\in A^\times$ for all $i,j$ represent the same element. We can write $c'=cd$ for $d\in A^\times$. Then there exist the unique $d$ (namely $d=b_1$) such that the first component of the tuple $t'$ is equal to $1$. Then $t'\coloneqq (1, (cb_ib_1c^{-1})_{i=2}^{p_e-2\chi(S)},(ca_jc^{-1},cb_1a_j'b_1^{-1}c^{-1})_{j=1}^{1-\chi(S)})$. The elements $p_i\coloneqq b_ib^{-1}\in A^\times$ for all $i\in\{2,\dots,p_e-2\chi(S)\}$, $q_j\coloneqq a_j\in A^\times$ and $q_j'\coloneqq b_1a'_jb_1^{-1}\in A^\times$ for all  $j\in\{1,\dots,1-\chi(S)\}$ determine the element in $(A^\times)^{p_e-2\chi(S)}\times (A^\times\times A^\times)^{1-\chi(S)}/(A^\times\times A^\times)$ uniquely up to common conjugation by $A^\times$. Therefore,
$$\bigl((A^\times)^{p_e-2\chi(S)}\times (A^\times\times A^\times)^{1-\chi(S)}\bigr)/(A^\times\times A^\times)\cong (A^\times)^{1+p_e-4\chi(S)}/A^\times.$$
This result agrees with the result from~\cite{KR} that was obtained using spectral networks and partial abelianization procedure.

\subsubsection{\texorpdfstring{$G$}{G} is \texorpdfstring{$\Sp_2(A,\sigma)$}{Sp_2(A,sigma)} or \texorpdfstring{$\OO_{1,1}(A,\sigma)$}{O_{1,1}(A,sigma)}} Let now $A$  be a unital $\R$-algebra endowed with an anti-involution, i.e.\ with an $\R$-linear map $\sigma\colon A\to A$ such that $\sigma(ab)=\sigma(b)\sigma(a)$ for all $a,b\in A$ and $\sigma^2=\Id$. We denote by $A^\sigma\coloneqq \{a\in A\mid \sigma(a)=a\}$ and
$A^{-\sigma}\coloneqq\{a\in A\mid \sigma(a)=-a\}$ (for more details about groups over noncommutative involutive algebras see~\cite{ABRRW,R-Thesis}).

Let $\Omega_s\coloneqq \begin{pmatrix} 0 & 1 \\ 1 & 0\end{pmatrix}$ and $\Omega_s\coloneqq \begin{pmatrix} 0 & 1 \\ -1 & 0\end{pmatrix}$. Then we define
$$\OO_{1,1}(A,\sigma)\coloneqq \left\{g\in\Mat_2(A)\mid \sigma(g)^t\Omega_s g=\Omega_s\right\},$$
$$\Sp_2(A,\sigma)\coloneqq \left\{g\in\Mat_2(A)\mid \sigma(g)^t\Omega_a g=\Omega_a\right\}.$$
For these two groups, we again take as $\Pp$ the upper triangular matrices. Then $\L=\{\diag(a,\sigma(a)^{-1})\mid a\in A^\times\}$ and $\Up_*=\left\{\begin{pmatrix}1 & b \\ 0 & 1 \end{pmatrix}\mid v\in (A^{-\sigma})^\times\right\}\cong (A^{-\sigma})^\times$ if $G=\OO_{1,1}(A,\sigma)$ and $\Up_*=\left\{\begin{pmatrix}1 & b \\ 0 & 1 \end{pmatrix}\mid v\in (A^{\sigma})^\times\right\}\cong (A^{\sigma})^\times$ if $G=\Sp_2(A,\sigma)$.

So we obtain that $\Rep^{\fr}(\pi_1(S),G)$ is homeomorphic to
$$((A^{-\sigma})^\times)^{p_e-2\chi(S)}\times (A^\times)^{1-\chi(S)}/A^\times\text{, if $G=\OO_{1,1}(A,\sigma)$},$$
$$((A^{\sigma})^\times)^{p_e-2\chi(S)}\times (A^\times)^{1-\chi(S)}/A^\times\text{, if $G=\Sp_2(A,\sigma)$}.$$

\subsubsection{\texorpdfstring{$G$}{G} is Hermitian of tube type that can be seen as \texorpdfstring{$\Sp_2(A,\sigma)$}{Sp_2(A,sigma)}} This example generalizes Example~\ref{ex:sp_2n}.

An involutive finite-dimensional $\R$-algebra is called \defin{Hermitian} if the space $A^\sigma_+=\{a^2\mid a\in (A^\sigma)^\times\}$ is an open proper convex cone in $A^\sigma$. For more details about Hermitian algebras and symplectic groups over them we refer the reader to~\cite{ABRRW}. If $(A,\sigma)$ is Hermitian, then the group $\Sp_2(A,\sigma)$ is a Hermitian Lie group of tube type.

In this case $\Upp$ is homeomorphic to the cone of squares $A^\sigma_+=\{a^2\mid a\in (A^\sigma)^\times\}$ in $A^\sigma$, and $\L_0=\L$ is isomorphic to $A^\times$. The group $U_{(A,\sigma)}=\{a\in A^\times\mid \sigma(a)a=1\}$ is a maximal compact subgroup of $A^\times$. It acts on itself by conjugation and on $A^\sigma_+$ by similarities, i.e.\ for $g\in A^\times$ and $B\in A^\sigma_+$, $g(B)=gB\sigma(g)$. 
Using this, we obtain that the space of maximal framed representations is homeomorphic to:
$$\bigl((A^\sigma_+)^{p_e-2\chi(S)}\times (A^\times)^{1-\chi(S)}\bigr)/A^\times=\bigl((A^\sigma_+)^{p_e-3\chi(S)}\times (U_{(A,\sigma)})^{1-\chi(S)}\bigr)/U_{(A,\sigma)}.$$
So we obtain the result of~\cite{R-Thesis}.

In the case $A=\Mat(n,\R)$ and $\sigma$ is the transposition, $U_{(A,\sigma)}=\OO(n)$ that has two connected component, so the space of maximal framed representations has $2^{1-\chi(S)}$ connected components. If $A=\Mat(n,\CC)$ or $A=\Mat(n,\HH)$ and $\sigma$ is the transposition composed with the complex resp. quaternionic conjugation, then $U_{(A,\sigma)}=\UU(n)$, resp. $U_{(A,\sigma)}=\Sp(n)$ that are connected, so the space of maximal framed representations is connected as well.

\subsection{\texorpdfstring{$G$}{G} is the connected adjoint form of a Hermitian Lie group of tube type} Not all Hermitian Lie groups of tube type can be seen as $\Sp_2(A,\sigma)$ for some involutive algebra $(A,\sigma)$. However, our approach can be applied also in general case for any adjoint form of Hermitian Lie groups of tube type.  In this case, $G$ is the isometry group of the tube domain $J+iJ_+$ where $J$ is a formally real Jordan algebra and $J_+$ is the cone of squares in $J$. In this case, $\Upp=J_+$ and $\L_0=\L$ is the subgroup of $\GL(J)$ preserving $J_+$ and $K$ is a maximal compact of $\L$ which is the subgroup of $\L$ that preserves $1$ (for more details see~\cite[Chapter~10]{Faraut}). Further, the polar decomposition for $\L$ gives us the identification $\L=J_+\times K$. Using the~classification of formally real Jordan algebras, we obtain that $\Rep_+^{\fr}(\pi_1(S),G)$ is homeomorphic to:
\begin{enumerate}
\item $(\Sym^+(n,\R)^{p_e-3\chi(S)}\times \PO(n)^{1-\chi(S)})/\PO(n)$ if $G=\PSp(2n,\R)$ for $n\geq 1$;
\item $(\Herm^+(n,\CC)^{p_e-3\chi(S)}\times \SU(n)^{1-\chi(S)})/\SU(n)$ if $G=\SU(n,n)$ for $n\geq 1$;
\item $(\Herm^+(n,\HH)^{p_e-2-3\chi(S)}\times \PSp(n)^{1-\chi(S)})/\PSp(n)$ if $G$ is the adjoint form of $\SO^*(4n)$  for $n\geq 1$;
\item $(\Herm^+(3,\Oc)^{p_e-3\chi(S)}\times F_4^{1-\chi(S)})/F_4$, where $F_4$ is the compact from, if $G=E_{7(-25)}$.
\item If $G=\SO_0(2,n+1)$ for $n\geq 2$, then $J$ is the formally real Jordan algebra of Clifford type of signature $(1,n)$. This case will be considered as a part of the next example.
\end{enumerate}
In cases (1)--(4), unless $G=\PSp(2n,\R)$ for $n$ even, then the maximal compact subgroup $K$ of the Levi factor is always connected, so the space $\Rep_+^{\fr}(\pi_1(S),G)$ is connected as well. If $G=\PSp(2n,\R)$ and $n$ is even, then $K$ has two connected components, and so $\Rep_+^{\fr}(\pi_1(S),G)$ has $2^{1-\chi(S)}$ connected components.

\subsection{\texorpdfstring{$G$}{G} is an indefinite orthogonal group \texorpdfstring{$\SO_0(1+p,n+p)$}{SO(1+p,n+p)} where \texorpdfstring{$p\geq 1$}{p>0}, \texorpdfstring{$n\geq 2$}{n>1}}
Let $G$ be the group $\SO_0(1+p,n+p)$ where $p,n\in\N$ and $n\geq 2$ considered as the connected component of the identity of the isometry group of the symmetric bilinear form $j_p(v_1,v_2)=\frac{1}{2}v_1^tJ_pv_2$ on $V\coloneqq \R^{n+2p+1}$ where:
$$J_p=\begin{pmatrix}
0 & 0 & Q_{p+1}^t \\
0 & -\Id_{n-1} & 0 \\
Q_{p+1} & 0 & 0
\end{pmatrix}$$
where $\Id_{n-1}$ is the identity $(n-1)\times(n-1)$-matrix and
$$Q_{p+1}=\begin{pmatrix}
0 & 0& \dots & 0 & 1 \\
0 & 0& \dots & -1 & 0 \\
0 & (-1)^{p}& \dots & 0 & 0 \\
(-1)^{p+1} & 0& \dots & 0 & 0
\end{pmatrix}.$$
We consider the flag variety:
$$\F=\{F^1\subset F^2\subset\dots\subset F^p\mid \dim(F^i)=i,\; J|_{l_i}=0\}.$$
\begin{df}\begin{itemize}
\item The basis $\mathbf{e}=(e_1,\dots,e_{n+2p+1})$  of $V$ is called \defin{normal} if the form $j_p$ in the basis $\mathbf{e}$ is given by the matrix $J_p$.
\item We say that a flag $F\in\F$ is \defin{generated} by a normal basis $\mathbf e$ if $F=(F^1\subset\dots\subset F^p)$ and $F^i=\Span(e_1,\dots,e_i)$. We write $F=\Flag(\mathbf{e})$. We denote $F_{st}=\Flag(\mathbf{e}_{st})$.
\end{itemize}
\end{df}
The group $G$ acts on $\F$ transitively. We denote by $\Pp$ the parabolic subgroup of $G$ stabilizing $\Flag(\mathbf e_{st})$. Therefore $\F$ can be seen as the following homogeneous space $G/\Pp$. By the Levi decomposition, $\Pp$ is a semidirect product of its unipotent subgroup $\Up$ and its Levi subgroup $\L$.

Let $\mathbf e_{st}=(e_1,\dots, e_{n+2p+1})$ be the standard basis of $V$. We denote by $W=\Span(e_{p+1},\dots,e_{p+n+1})$. This is a vector subspace of $V$ and the restriction of $j_p$ to $W$ agrees with $j_0$ and has signature $(1,n)$. Therefore, the set of positive vectors $\{w\in W\mid\ j_0(w,w)>0\}$ is the union of two open proper convex cones. We chose one of them and call it $\Omega\subset W$. We also choose a vector $w_0\in \Omega$.

We now describe the Levi subgroup $\L$ corresponding to $\Pp$. We consider the group $(\R^\times)^p\times\SO(1,n)$. Let $x=(x_1,\dots,x_p,y)\in (\R^\times)^p\times\SO(1,n)$. We we call the signature of $x$ the value $\sgn(x)\coloneqq \prod_{i=1}^p\sgn(x_i)\sgn(y)\in\{1,-1\}$, where $\sgn(y)=1$ if $y\in\SO_0(1,p)$ and $\sgn(y)=-1$ otherwise.
We identify the Levi subgroup with the subgroup of $(\R^\times)^p\times\SO(1,n)$ consisting of all elements of signature $1$ in the following way:
\begin{equation}\label{Levi-identification}
(x_1,\dots,x_p,y)\mapsto \diag(x_1,\dots,x_p,y,x_p^{-1},\dots,x_1^{-1}).
\end{equation}

The positive semigroup $\Upp$ of $\Up$ is generated by the elements $u_1(c_1),\dots u_{p-1}(c_{p-1})$ and $u_{p}(v)$ where $c_1,\dots,c_{p-1}>0$, $v\in\Omega$, $u_i(c_i)$ has $1$ on the diagonal, $c_i$ in the $(i,i+1)$- and $(n+2p+1-i,n+2p+2-i)$-positions and zeroes everywhere else and
$$u_{p}(v)=\begin{pmatrix}
\Id_{p-1} & 0 & 0 & 0\\
0 & 1 & v^t & j_0(v,v) & 0\\
0 & 0 & \Id_{n+1} & J_0v & 0\\
0 & 0 & 0 & 1 & 0\\
0 & 0 & 0 & 0 & \Id_{p-1}
\end{pmatrix}.$$
The stabilizer $\L_0$ of $\Upp$ in $\L$ under the identification~(\ref{Levi-identification}) is
$$\L_0=\{x=(x_1,\dots,x_p,y)\mid \sgn(x_1)=\dots=\sgn(x_p)=\sgn(y),\;\sgn(x)=1\}.$$
\begin{rem}
If $p$ is even, $\L_0$ is connected. Otherwise, it has two connected components.
\end{rem}

A maximal compact subgroup $K$ of $\L_0$ can be described as follows:
\begin{enumerate}
    \item if $p$ is odd, then
    $$K=\{(\varepsilon,\dots,\varepsilon,y)\mid y\in \OO(w_0^{\bot_{j_0}}),\;\varepsilon=\sgn(y)\}\cong\OO(w_0^\bot)\cong\OO(n)$$
    where $\OO(w_0^{\bot_{j_0}})$ acts on $W$ as follows: $y\in\OO(w_0^{\bot_{j_0}})$, $y(w_0+v)=\sgn(y) w_0+y(v)$ for $v\in w_0^{\bot_{j_0}}$;

    \item If $p$ is even, then
    $$K=\{(1,\dots,1,y)\mid y\in \SO(w_0^\bot)\}\cong\SO(w_0^\bot)\cong\SO(n).$$
\end{enumerate}

Let $C\in\Mat_{p\times(p-1)}(\R)$ and $\mathbf{v}=(v_1,\dots,v_p)$ for all $v_i\in W$. We denote
$$u(C,\mathbf{v})\coloneqq\prod_{i=1}^{p}\left(\prod_{j=1}^{p-1}u_j(c_{ij})\right)u_p(v_i)\in \Up.$$

\begin{fact}
\begin{itemize}
\item Every element $u\in \Up$ can be written in a unique way as $u=u(C,\mathbf{v})$ where $C\in\Mat_{p\times(p-1)}(\R)$ and $\mathbf{v}=(v_1,\dots,v_p)\in W^p$.

\item Every element $u\in \Upp$ can be written in a unique way as $u=u(C,\mathbf{v})$ where $C\in\Mat_{p\times(p-1)}(\R_+)$ and $\mathbf{v}\in\Omega^p$.
\end{itemize}
\end{fact}

The space $\Rep_+^{\fr}(\pi_1(S),G)$ is homoeomorphic to:
\begin{align*}
& \left(\R_+^{(p_e-2\chi(S))\cdot p(p-1)}\times\Omega^{(p_e-2\chi(S))\cdot p}\times\L^{1-\chi(S)}\right)/\L\\
\cong\; & \R_+^{((p_e-2\chi(S))\cdot p-\chi(S))(p-1)}\times \left(\Omega^{(p_e-2\chi(S))\cdot p-\chi(S)}\times K^{1-\chi(S)}\right)/K.
\end{align*}
If $p$ is even this space is connected because in this case $K\cong\SO(n)$. Otherwise, $K\cong\OO(n)$ and this space has $2^{1-\chi(S)}$ connected components.

\subsection{\texorpdfstring{$G$}{G} is an exceptional from the family \texorpdfstring{$F_{4(4)}$}{F4}, \texorpdfstring{$E_{6(2)}$}{E6}, \texorpdfstring{$E_{7(-5)}$}{E7}, \texorpdfstring{$E_{8(-24)}$}{E8}}
In this example we use results on exceptional Lie groups from~\cite{Yokota}. In this case $\Theta = \{\alpha_1, \alpha_2\}$, $W(\Theta)=W_{G_2}=\langle s_1,s_2 \rangle$ with the longest element $(s_1s_2)^3$ and the relation $(s_1s_2)^3=(s_2s_1)^3$. The Levi subgroup $\L_0$, its maximal compact subgroup $K$ and the positive semigroup $\Upp$ can be explicitly determined: $\L_0\cong\R_+\times\R_+\times H$ and $\Upp\cong \R_+^3\times C^3$ where
\begin{enumerate}
    \item for $G=F_{4(4)}$ (split form), $H=\PGL_3(\R)$, $K=\SO(3)$, $C=\Sym^+_3(\R)$;
    \item for $G=E_{6(2)}$ (quasi-split form), $H=\PGL_3(\CC)$, $K=\PU(3)$, $C=\Herm^+_3(\CC)$;
    \item for $G=E_{7(-5)}$, $H=\PGL_3(\HH)$, $K=\PSp(3)$, $C=\Herm^+_3(\HH)$;
    \item for $G=E_{8(-24)}$, $H=E_{6(-26)}$ the connected adjoint form, $K=F_4$ the connected compact adjoint form, $C=\Herm^+_3(\Oc)$.
\end{enumerate}
The Levi factor $\L_0$ acts on $\Upp$ as follows: $\L_0\times\Upp\to \Upp$
$$((a_1,a_2,b),(x_1,x_2,x_3,y_1,y_2,y_3))\mapsto (a_1a_2x_1,a_1a_2x_2,a_1a_2x_3,a_2(b.y_1),a_2(b.y_2),a_2(b.y_3)),$$
where all $a_i,x_i\in \R_+$, all $y_i\in C$, $b\in H$ and $b.y_i$ denotes the action of $b$ on $y_i\in C$.

The representation space $\Rep_+^{\fr}(\pi_1(S),G)$ is homeomorphic to

\begin{enumerate}
    \item $\R_+^{3 p_e-8\chi(S)}\times (\Sym^+_3(\R)^{3 p_e-7\chi(S)}\times \SO(3)^{1-\chi(S)})/\SO(3)$ for $G=F_{4(4)}$;
    \item $\R_+^{3 p_e-8\chi(S)}\times (\Herm^+_3(\CC)^{3 p_e-7\chi(S)}\times \PU(3)^{1-\chi(S)})/\PU(3)$ for $G=E_{6(2)}$;
    \item $\R_+^{3 p_e-8\chi(S)}\times (\Herm^+_3(\HH)^{3 p_e-7\chi(S)}\times \PSp(3)^{1-\chi(S)})/\PSp(3)$ for $G=E_{7(-5)}$;
    \item $\R_+^{3 p_e-8\chi(S)}\times (\Herm^+_3(\Oc)^{3 p_e-7\chi(S)}\times F_4^{1-\chi(S)})/F_4$ for $G=E_{8(-24)}$.
\end{enumerate}

In all these cases $K$ is connected, so $\Rep_+^{\fr}(\pi_1(S),G)$ is connected as well.

\bibliographystyle{alpha} \bibliography{bibl}
\end{document}